\documentclass[preprint, 11pt]{elsarticle}
\usepackage[top=2.3cm, bottom=2.3cm, left=2.6cm, right=2.6cm]{geometry}
\usepackage{array}
\usepackage{float}
\usepackage{units}
\usepackage{multirow}
\usepackage{amsthm}
\usepackage{amsmath}
\usepackage{amssymb}
\usepackage{graphicx}
\usepackage{epsfig}
\usepackage{color}
\usepackage[labelformat=simple]{subcaption}
\usepackage{footnote}
\makesavenoteenv{tabular}
\usepackage{setspace}
\graphicspath{{./Figs/}}

\usepackage{caption}
\makeatletter

\newtheorem{rem}{Remark}
\newtheorem{prop}{Proposition}
\newtheorem{lem}{Lemma}

\newcommand{\ie}{{\emph{i.e.}}}
\newcommand{\eg}{{\emph{e.g.}}}

\makeatother

\begin{document}

\begin{frontmatter}

\title{Energy-efficient Rail Guided Vehicle Routing for\\
Two-Sided Loading/Unloading Automated Freight Handling System}


\author[a]{Wuhua Hu}
\ead{hwh@ntu.edu.sg}
\author[b]{Jianfeng Mao\corref{corr}}
\ead{jfmao@ntu.edu.sg}
\author[c]{Keji Wei}
\ead{keji.wei.th@dartmouth.edu}

\address[a]{School of Electrical and Electronic Engineering, Nanyang Technological University, Singapore}
\address[b]{Division of Systems and Engineering Management,
Nanyang Technological University, Singapore}
\address[c]{Thayer School of Engineering, Dartmouth College, NH, United States}

\cortext[corr]{Corresponding author. Tel: +65 6790 5522; Fax: +65 6792 4062.}

\begin{abstract}
Rail-guided vehicles (RGVs) are widely employed in automated freight handling system (AFHS) to transport surging air cargo. Energy-efficient routing of such vehicles is of great interest for both financial and environmental sustainability. Given a multi-capacity RGV working on a linear track in AFHS, we consider its optimal routing under two-sided loading/unloading (TSLU) operations, in which energy consumption is minimized under conflict-avoidance and time window constraints. The energy consumption takes account of routing-dependent gross weight and dynamics of the RGV, and the conflict-avoidance constraints ensure conflict-free transport service under TSLU operations. The problem is formulated as a mixed-integer linear program, and solved by incorporating valid inequalities that exploit structural properties of the problem. The static problem model and solution approach are then integrated with a rolling-horizon approach to solve the dynamic routing problem where air cargo enters and departs from the system dynamically in time. Simulation results suggest that the proposed strategy is able to route an RGV to transport air cargo with an energy cost that is considerably lower than one of the most common heuristic methods implemented in current practice.
\end{abstract}


\end{frontmatter}

\section{Introduction}

Automated freight handling system (AFHS) is a type of automated material handling system (AMHS), which is widely adopted in facilities with massive material handling requests, such as freight terminals, distribution centers and production plants, to enhance system efficacy by minimizing operating cost and risk of human errors. As operating cost of AMHS can represent up to 70\% of the cost of a product \citep{liu2002design,giordano2008integrated}, it is critical to smartly design and operate AMHS to improve the overall economic and environmental performance. There has been a considerable growth of interest in studying such problems in both industrial and academic contexts.

This work considers improving AFHS installed in a freight terminal, which employs railed-guided vehicles (RGVs) to transport a tremendous amount of inbound and outbound cargo from their origins to destinations distributed along a linear track. The workload, which is especially high at its peak hour around midnight, leads to a conflict-prone environment that poses great challenges to terminal operations. The present AFHS has been developed to improve the terminal's throughput while eliminating potential human errors. However, the design is not optimal especially in terms of energy efficiency. As energy consumption constitutes one of major sources of operating cost and has gained increasing attention for enabling a greener and sustainable earth, improving energy efficiency of AFHS and in particular developing an energy-efficient RGV routing strategy is of great interest to the industry. So far only heuristic methods have been employed to route RGVs in the current system, which motivates us to develop a rigorous mathematical programming method to improve the system performance. The new routing strategy also meets various service requirements such as delivery time windows (TWs) and avoidance of unloading deadlocks, etc., by incorporating them into the mathematical program from which the strategy results. This is contrast to existing methods which achieve that by abruptly compromising the system's performance or by means of a posteriori sophisticated supervisory control \citep{giordano2008integrated}.

\subsection{Problem description}

A typical work area of the AFHS under consideration is depicted in Figure \ref{fig: toy example - scenario}. An RGV is operated over a linear track to transport containers between work stations located along both sides of the track. Containers are queued at work stations, and will be picked up and delivered to their destined stations on either sides of the track by the RGV via so-called \emph{two-sided loading/unloading} (TSLU) operations. The RGV and work stations are equipped with roller decks to support the TSLU operations. When the RGV is docked to a station, the roller decks rotate forward or backward accordingly to load or unload containers from or to either side of the track. The RGV can carry multiple containers subject to certain capacity limit.

Each transportation of a container is initiated with a pickup and delivery (PD) request to the central control system. Midway drop-off is not allowed in the current practice because of the substantial overhead caused by frequent acceleration and deceleration of the RGV. Once a container is picked up, it remains on the RGV until being delivered to its destination. As containers enter and depart from the AFHS dynamically, the control system accumulates unfinished PD requests and aims at routing the RGV to pick up and deliver the associated containers in an optimal sequence so as to minimize energy consumption required for completing all transport requests subject to service quality and conflict-avoidance constraints for smooth operations.

\begin{figure}[H]
\begin{centering}
\includegraphics[scale=0.9]{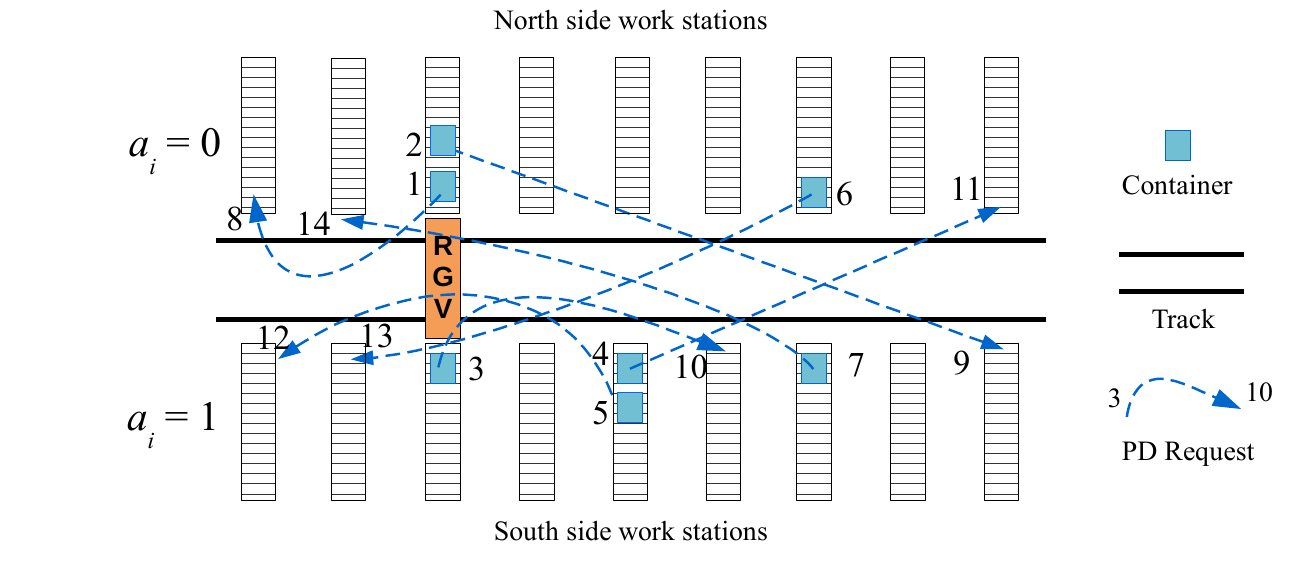}
\par
\end{centering}
\caption{Typical work area of an RGV in AFHS with exemplary PD requests.\label{fig: toy example - scenario}}
\end{figure}

\subsection{Related literature}

Sequencing PD tasks to be handled by a single vehicle is referred to as vehicle routing or job sequencing \citep{vis2006survey,roodbergen2009survey,carlo2014storage}, and can be treated as a pickup and delivery problem (PDP) \citep{berbeglia2007static,parragh2008survey}. The problem is NP-hard in general due to a complex combinatorial nature, and it includes our problem under consideration as a sophisticated case. Our problem is further complicated by dynamic arrivals of PD requests. So far there have been a variety of research investigations on static/dynamic vehicle routing problems (VRPs) for different applications \citep{psaraftis1988dynamic,berbeglia2010dynamic,pillac2013review}. The literature confined to a single vehicle is briefly reviewed as follows.

Static VRPs assume that all transport requests are known a priori. Atallah and Kosaraju \citep{atallah1987efficient} proved that the problem is polynomial-time solvable if the vehicle has unit capacity while confined to a linear track when no precedence and TW constraints are imposed on the transport requests. The same problem turns out to be NP-hard if the track topology changes to be a tree (or general graph) \citep{frederickson1993nonpreemptive,frederickson1976approximation}. If the vehicle has multiple capacity, the problem is NP-hard even if the track is a simple path \citep{guan1998routing}. Another closely related problem, the PDP (for goods transportation) or the dial-a-ride problem (DARP, for passenger transportation) which includes TW constraints on PD requests has been well studied. Algorithms based on branch-and-cut or column generation are available for solving the problem of a small to medium size \citep{parragh2008survey,cordeau2008recent}. Other related literature is on request/job sequencing in automated storage and retrieval systems (AS/RS), as referred to \citep{roodbergen2009survey} for a recent review.

The aforementioned literature all considered PDPs without loading constraints. In many applications, however, constraints also appear on loading \citep{iori2010routing}. In the \emph{traveling salesman problem with pickup and delivery and LIFO loading} (TSPPDL), a single vehicle must serve paired PD requests while both pickup and delivery must be performed in LIFO (last in first out) order. Heuristic algorithms for solving this problem were introduced in \citep{ladany1984optimal,carrabs2007variable}, and exact formulations and solutions were reported in \citep{carrabs2007additive,cordeau2010branch} which rely on tailored branch-and-cut algorithms. Another related problem is the \emph{traveling salesman problem with pickup and delivery and FIFO loading} (TSPPDF) in which both pickup and delivery must be performed in FIFO (first in first out) order. Heuristic algorithms were introduced in \citep{erdougan2009pickup} for solving the problem, and exact solutions were explored by using additive branch-and-bound \citep{carrabs2007additive} and tailored branch-and-cut \citep{cordeau2010branchFIFO}, respectively.

In practice, VRPs are dynamic in nature as PD requests appear stochastically in time. Research work on dynamic VRPs assumes the requests arrive in an unknown or known stochastic process \citep{berbeglia2010dynamic,pillac2013review}. With an unknown arrival process, deterministic algorithms have been proposed to treat dynamic VRPs and their performances are evaluated by competitive analysis (which compares the worst-case cost achieved by the algorithm for a dynamic problem and the optimal cost obtained for its counterpart having all requests being known a priori) \citep{ascheuer2000online,feuerstein2001line,berbeglia2010dynamic}. However, the algorithms and their analyses all rely on the assumption that the dynamic VRPs do not have side constraints like loading, precedence and TW constraints. On the other hand, if the arrival process of PD requests is known, strategies based on sampling or Monte Carlo simulations are often used to handle dynamic VRPs \citep{berbeglia2010dynamic}. Rolling-horizon based algorithms can also be employed if deterministic estimates of future requests are available \citep{psaraftis1988dynamic}. Alternatively, heuristic algorithms may be developed by exploiting structural properties of a problem, as referred to \citep{berbeglia2010dynamic} for a relevant review.

In particular, among existing literature, \eg, \citep{lau2006joint,ou2010scheduling,tang2010improving,Hu2012IEEM}, investigating operational aspects of AFHS (whereas investigations on higher-layer issues can be referred to \citep{derigs2013air} and the references therein), \citep{Hu2012IEEM} studied a most relevant but simpler problem, in which the RGV is assumed to have unit capacity and the goal is to minimize RGV's travel distance for completing all transport requests. The problem is a dynamic PDP with TW, FIFO queuing, and PD precedence constraints, and the investigation forms a pilot study towards obtaining a model and solution for the more complex routing problem considered in this work, of which partial results were reported in \citep{hu2013energy}.

This work differs from existing literature in several aspects. \emph{Firstly}, the problem under investigation is a capacitated PDP under unique conflict-free service constraints. We completely characterize these service constraints under TSLU operations, which include the well-known LIFO and FIFO service constraints as two special cases. This is the first time that such kind of characterization has become available for transport service under TSLU operations, to the best of our knowledge. \emph{Secondly}, the RGV routing problem aims to minimize total energy consumption of operating an RGV for completing all PD requests, which meets well with the interest of saving energy and reducing carbon emission in practice. This differs from the existing literature where travel distance or makespan (\ie, task completion time) is taken as the objective \citep{berbeglia2007static,parragh2008survey,xin2014energy}.  \emph{Thirdly}, structural properties of the new problem are exploited to reduce the problem domain and derive useful valid inequalities for improving the computational efficiency. \emph{Fourthly}, a rolling-horizon approach is developed for treating the dynamic RGV routing problem, where a way of handling non-zero initial conditions (\ie, the RGV starts with nonempty load) is introduced and explained in detail. The new issues revealed and solved all explain the challenges of enabling energy-efficient routing of an RGV confronted in a real AFHS.

The rest of the work is organized as follows. Section \ref{sec: problem formulation} perceives the static RGV routing problem as a sophisticated PDP problem, characterizes its conflict-free service requirements, and develops a full optimization model for it. Section \ref{sec: solution approach} reformulates the initial model into a mixed-integer linear program (MILP), reduces its domain by removing infeasible solutions, and derives valid inequalities from the problem structure for expediting the solution process. Section \ref{sec: rolling-horizon-appro} presents a rolling-horizon approach to treat the dynamic problem, as followed by comprehensive computational studies performed on random instances in Section \ref{sec: computational studies}. Finally, conclusions are drawn in Section \ref{sec: conclusions}. Supporting materials that are helpful to understand the main results are collected in the Appendices.

\section{Static Routing Problem Formulation} \label{sec: problem formulation}

A multi-capacity RGV is routed to serve $n$ PD requests with TW constraints on the deliveries. As midway drop-off is not allowed in current practice, fulfilling a PD request is equivalent to completing two tasks, pickup task (\ie, loading a container from its origin) and delivery task (\ie, unloading the container at its destination). Therefore, serving $n$ PD requests is equivalent to completing $2n$ tasks and the optimal routing solution is a processing order of the $2n$ tasks for the RGV to consume least energy under service quality and feasibility constraints.

To facilitate problem formulation, we assign a unique integer ID to each task. A pickup task is associated with an integer $i$, where $1\le i \le n$, and its corresponding delivery task associated with $i+n$. So the $i$th PD request means a pair of directed tasks $i\dashrightarrow i+n$, where $\dashrightarrow$ means a simple path that may contain multiple arcs. We denote the sets of pickup and delivery tasks as $P$ and $D$, respectively, \ie, $P = \{1, 2, \dots, n\}$ and $D = \{n+1, n+2, \dots, 2n\}$. In the example shown in Figure \ref{fig: toy example - scenario}, there are seven PD requests indicated by dashed arrows, in which $P=\{1, 2, \dots,7\}$, $D=\{8, 9, \dots, 14\}$ and the PD requests correspond to seven pairs of PD tasks $\{1\dashrightarrow 8, 2\dashrightarrow 9, \dots, 7\dashrightarrow 14\}$. For convenience, we refer to the PD requests by their pickup tasks, as $P$.

The static RGV routing problem is then modeled by
a directed graph $G=(V, A)$, where a vertex $i$ in $V$ represents
a pickup or delivery task with the ID of $i$ and an arc $(i,j)$ in $A$ represents
a plausible processing order between the two tasks $i$ and $j$,
\ie, whether the task $j$ could be processed right after the task $i$. Two virtual tasks, $0$ and $2n+1$, are added to represent the start and end positions of the RGV, respectively. Define $V' = P \cup D = \{1,2,\dots,2n\}$ and $V = \{0\} \cup V' \cup \{2n+1\}$. The arcs $(i,j)$ for all $i,j\in V'$ and $i,j\in V$ give rise to arc sets $A'$ and $A$, respectively (both of which can be reduced by removing infeasible arcs in Section \ref{subsec: arc_set_reduction}).

A feasible routing solution is then a path starting from vertex $0$, going through all vertices in $V'$ exactly once and ending at vertex $2n+1$. While a path can generally be formulated by a sequence of binary decision variables $x_{ij}$ which indicate the arcs $(i,j)$ used in the path, another group of binary decision variables $y_{ij}^k$, which indicate the arcs traversed during the service of each request $k$ is also required for enforcing conflict-free service constraints as revealed next. The new problem and its unique model differentiate our work from existing literature.

\subsection{Conflict-free service constraints}

In practice there are physical constraints, such as containers on an RGV cannot swap positions, that restrict the TSLU operations. The multi-capacity RGV thus must be routed to load and unload containers in a valid order in order to avoid infeasible operations and deadlock. Here, \emph{infeasible operation} means an operation that attempts to unload a container before another one which is however impossible to realize in practice for physical constraints; and \emph{deadlock} means that a container cannot be unloaded without temporally dropping off another container and it remains so if the order of unloading the two containers is reversed. The requirements are equivalent to imposing conflict-free service constraints on the TSLU operations. Before describing these constraints, we classify all PD requests to be handled by the TSLU operations into four types:
\begin{align}
P^{1} & \triangleq \{i\in P:\,a_{i}=a_{i+n}=0\},\,P^{2}\triangleq\{i\in P:\,a_{i}=a_{i+n}=1\},\nonumber \\
P^{3} & \triangleq \{i\in P:\,a_{i}=0,\,a_{i+n}=1\},\,P^{4}\triangleq\{i\in P:\,a_{i}=1,\,a_{i+n}=0\},\label{eq: PD types}
\end{align}
where $a_i$ is an indicator variable to show on which side task $i$ occurs, and it indicates the north side if $a_i = 0$ and the south side otherwise. The four types of PD requests are illustrated in Figure \ref{fig:Four-types-of PD requests}. For the example shown in Figure \ref{fig: toy example - scenario}, we have $P^1 = \{1\}$, $P^2 = \{3,5\}$ $P^{3} = \{2,6\}$ and $P^4 = \{4, 7\}$.
\begin{figure}[H]
\begin{centering}
\includegraphics[scale=0.6]{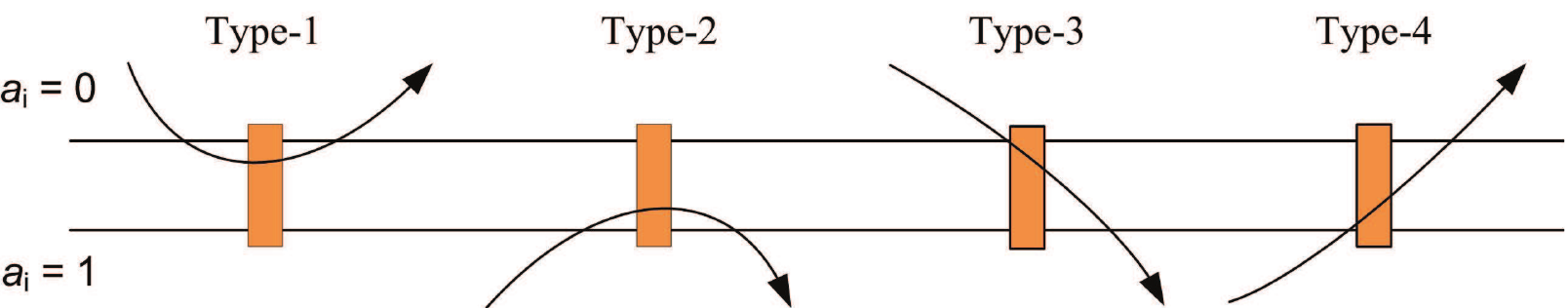}
\par\end{centering}
\caption{Four types of PD requests.\label{fig:Four-types-of PD requests}}
\end{figure}

Various conflicts may arise if tasks belonging to different types of PD requests are processed sequentially. The types of conflicts reviewed in the literature (\eg, \citep{cordeau2010branch,cordeau2010branchFIFO,erdougan2009pickup}) can be regarded as loading/unloading systems restricted to one type of the PD requests defined above. Thus, their conflict-free service constraints can be satisfied by sticking to a certain task service order: If PD requests are all of Type-1 (or Type-2), the LIFO service order will suffice; and if PD requests are all of Type-3 (or Type-4), the FIFO service order will do.

Under TSLU operations, however, a single LIFO or FIFO service order cannot guarantee conflict-free services. As there are four mixed types of PD requests, we will face 10 groups of scenarios for serving different combinations of PD requests. Among them, one group is always conflict-free in which tasks of different types of requests can be served in any order, and that the rest 9 groups have to be constrained to avoid conflicts. Specifically, the conflict-free scenarios are concerned with service of PD request pairs $P^1 \sim P^2$, and the rest 9 groups of scenarios are concerned with service of PD request pairs $P^1 \sim P^1/P^3/P^4$, $P^2 \sim P^2/P^3/P^4$, $P^3 \sim P^3/P^4$, $P^4 \sim P^4$, where the symbol / means ``or''. Different constraints may be imposed in the 9 groups of scenarios to ensure conflict-free services.

The ten groups of scenarios and their associated service constraints can be classified into six cases below. The first two cases correspond to conventional scenarios with LIFO or FIFO loading restrictions, and the next three cases are unique to the routing problem under consideration, and the last case describes conflict-free scenarios where no special service constraint is required.

\subsubsection{Case 1. $P^1\sim P^1$ and $P^2\sim P^2$: LIFO service}

As illustrated in Figure \ref{fig:LIFO}, LIFO service order should be maintained between any pair of PD requests in $P^1$ (or $P^2$ alike) for conflict avoidance. Given requests $j,k \in P^1$ which share the RGV for a certain period of time, if the pickup task $j$ is processed after the pickup task $k$, then the delivery task $j+n$ must be processed before the delivery task $k+n$.

\begin{figure}[H]
\begin{centering}
\includegraphics[scale=0.6]{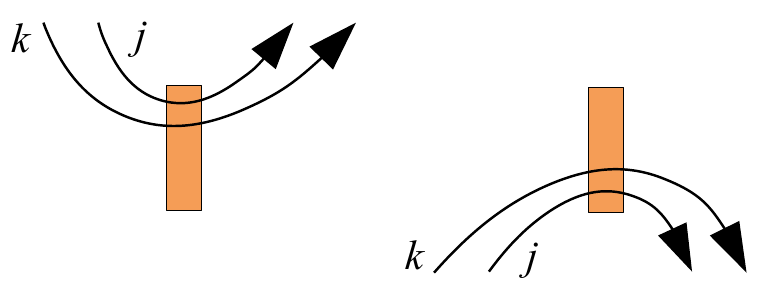}
\par\end{centering}
\caption{LIFO service: $P^1\sim P^1$ and $P^2\sim P^2$. \label{fig:LIFO}}
\end{figure}

To characterize this mathematically, we introduce binary decision variables $y_{ij}^k$ for all request $k\in P$ and arc $(i,j)$ in a feasible set. We have $y_{ij}^k = 1$, if arc $(i,j)$ is traversed on the path from vertex $k$ to vertex $k+n$ (\ie, during the service for request $k$); and $y_{ij}^k = 0$ otherwise. The LIFO service order is enforced by the following constraints:
\begin{equation}
\sum_{i: (i,j)\in A'}y_{ij}^k=\sum_{i: (i,j+n)\in A'}y_{i,j+n}^k, \quad
\begin{cases}
\forall   j\in P^{1}\backslash\{k\}, k\in P^{1},\\
\forall   j\in P^{2}\backslash\{k\}, k\in P^{2}.
\end{cases}\label{eq:LIFO}
\end{equation}
Constraint \eqref{eq:LIFO} means that if task $j$ is processed between $k$ and $k+n$, then task $j+n$ must also be processed between $k$ and $k+n$, namely, the LIFO service order is implemented.

\subsubsection{Case 2. $P^3\sim P^3$ and $P^4\sim P^4$: FIFO service}

Similarly, FIFO service order should be maintained between any pair of PD requests in $P^3$ (or $P^4$ alike) for conflict avoidance, as illustrated in Figure \ref{fig:LIFO}. This requirement can be met by enforcing constrain (\ref{eq:FIFO}) below, which means that if a pickup task of request $k$ in $P^3$ (or $P^4$) is processed before the service of request $j$ of the same type, then the delivery task of request $k$ must be processed before the completion of request $j$, \ie, the FIFO service order is enforced.

\begin{figure}[H]
\begin{centering}
\includegraphics[scale=0.6]{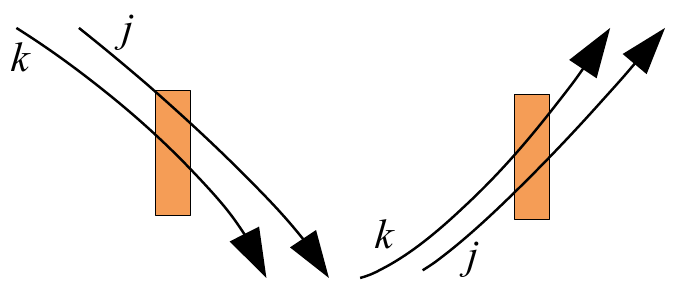}
\par\end{centering}
\caption{FIFO service: $P^3\sim P^3$ and $P^4\sim P^4$. \label{fig:FIFO}}
\end{figure}

\begin{equation}
\sum_{i: (i,j)\in A'}y_{ij}^k+\sum_{i: (i,j+n)\in A'}y_{i,j+n}^k\le1,\quad \begin{cases}
\forall   j\in P^{3}\backslash\{k\}, k\in P^{3},\\
\forall   j\in P^{4}\backslash\{k\}, k\in P^{4}.
\end{cases}\label{eq:FIFO}
\end{equation}

\subsubsection{Case 3. $P^1\sim P^3$ and $P^2\sim P^4$: crossing first in (CFI) service}

This case is illustrated in Figure \ref{fig:Crossing-First-In}. Assume $j \in P^3, k \in P^1$ and they are simultaneously served by an RGV in a certain time interval. The PD request $j$ is a crossing request whose pickup and delivery locations are on opposite sides of the track. If the pickup task $j$ is processed after the pickup task $k$ but before the delivery task $k+n$, then the delivery task $j+n$ cannot be performed because the container associated with the request $k$ will block the way out. Thus, in this case, the crossing request must be handled first for ensuring conflict-free service, leading to the so-called CFI service order. This applies to the case for $j\in P^{4}, k\in P^{2}$ alike.

\begin{figure}[H]
\begin{centering}
\includegraphics[scale=0.6]{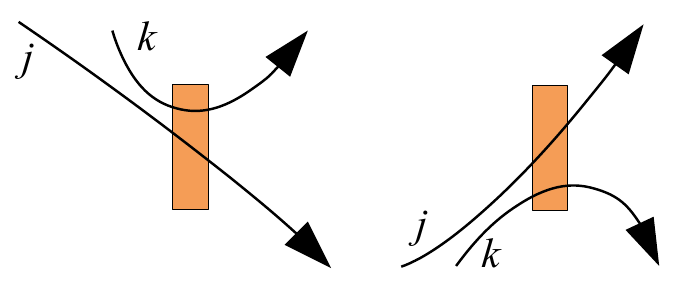}
\par\end{centering}
\caption{CFI service: $P^1\sim P^3$ and $P^2\sim P^4$. \label{fig:Crossing-First-In}}
\end{figure}

The CFI service order can be enforced by the following constraints:
\begin{equation}
y_{ij}^k = 0, \forall (i,j)\in A',\quad  \begin{cases}
\forall   j\in P^{3}, k\in P^{1},\\
\forall   j\in P^{4}, k\in P^{2}.
\end{cases}\label{eq:CFI}
\end{equation}
Also note that, for $j\in P^{1}, k\in P^{3}$ (or $j\in P^{2}, k\in P^{4}$), we can obtain another set of conflict-free service constraints, $\sum_{i: (i,j+n)\in A'}y_{i,j+n}^k \le \sum_{i: (i,j)\in A'}y_{ij}^k$, meaning that if task $j+n$ is processed between tasks $k$ and $k+n$, then task $j$ must also have been processed. As shown in Appendix \ref{apdix: CFI-alternative}, this set of constraints are equivalent to the constraints in (\ref{eq:CFI}) and hence can be ignored.

\subsubsection{Case 4. $P^1\sim P^4$ and $P^2\sim P^3$: crossing last out (CLO) service}

Similarly, CLO service order should be maintained in Case 4 as illustrated in Figure \ref{fig:Crossing-Last-Out}. Assume that $j \in P^4, k \in P^1$ and they are simultaneously served by the RGV in a certain time interval. Request $j$ is still the crossing request. This time, the pickup tasks $j$ and $k$ can be processed in a flexible order, but the delivery task $j+n$ must be processed after the delivery task $k+n$ because the container associated with request $j$ is queued behind the one associated with request $k$ on the RGV.

\begin{figure}[H]
\begin{centering}
\includegraphics[scale=0.6]{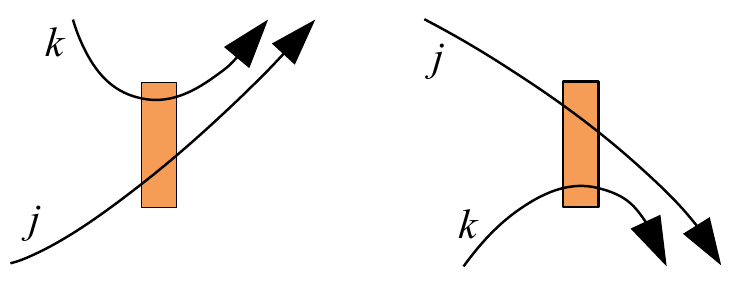}
\par\end{centering}
\caption{CLO service: $P^1\sim P^4$ and $P^2\sim P^3$. \label{fig:Crossing-Last-Out}}
\end{figure}

The CLO service order can be enforced by the following constraints:
\begin{align}
y_{i,j+n}^k=0, \forall (i,j+n)\in A', &\quad \begin{cases}
\forall   j\in P^{4}, k\in P^{1},\\
\forall   j\in P^{3}, k\in P^{2}.
\end{cases} \label{eq:CLO}
\end{align}
As in the CFI case, for all $j\in P^{1}, k\in P^{4}$ (or $j\in P^{2}, k\in P^{3}$) we can obtain another set of conflict-free service constraints, $\sum_{i: (i,j)\in A'}y_{ij}^k \le \sum_{i: (i,j+n)\in A'}y_{i,j+n}^k$, which are equivalent to the constraints in (\ref{eq:CLO}) and hence can be ignored. A proof of this fact is again referred to Appendix \ref{apdix: CFI-alternative}.

\subsubsection{Case 5. $P^3\sim P^4$: Deadlock}

This case corresponds to the deadlock case for any pair of PD requests $P^3\sim P^4$ as illustrated in Figure \ref{fig:Deadlock}. Such pair of PD requests should never be simultaneously served by the RGV because the corresponding containers will block their way out from each other.

\begin{figure}[H]
\begin{centering}
\includegraphics[scale=0.6]{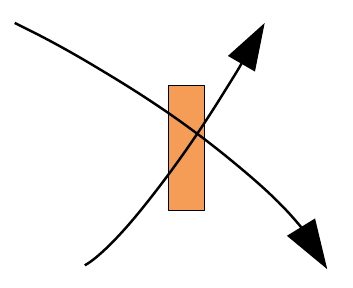}
\par\end{centering}
\caption{Deadlock: $P^3\sim P^4$. \label{fig:Deadlock}}
\end{figure}

The deadlock can be avoided by enforcing the following constraint, which completely avoids overlap between services of the aforementioned two types of PD requests:
\begin{align}
y_{ij}^k = 0, \forall (i,j)\in A', \quad \begin{cases}
\forall   j\in P^{4}, k\in P^{3},\\
\forall   j\in P^{3}, k\in P^{4}.
\end{cases} \label{eq:Deadlock-a}
\end{align}
It can be shown that the above condition implies that $y_{i,j+n}^k = 0$, for all $(i,j+n)\in A'$ with $j, k$ in the domains given in \eqref{eq:Deadlock-a}.

\subsubsection{Case 6. $P^1\sim P^2$: Free case}

This case corresponds to the free case illustrated in Figure \ref{fig:Free}, in which any pair of PD requests do not interfere from each other. The requests can thus be freely served by the RGV with no dedicated service constraints.

\begin{figure}[H]
\begin{centering}
\includegraphics[scale=0.6]{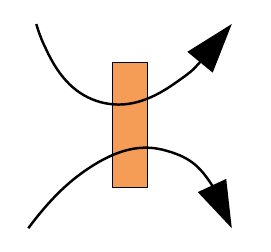}
\par\end{centering}
\caption{Free service: $P^1\sim P^2$. \label{fig:Free}}
\end{figure}

\begin{rem}
The LIFO and FIFO service constraints formulated in (\ref{eq:LIFO}) and (\ref{eq:FIFO}) have forms conciser than those developed in \citep{erdougan2009pickup} where two additional groups of binary variables with the same dimensionality of $y_{ij}^k$ are required. The current formulation avoids those binary variables and the constraints associated, and hence is computationally more efficient in general.
\end{rem}

\subsection{Problem formulation}

This subsection develops a mathematical model of the energy-efficient RGV routing problem. The major parameters and notations used are listed below for ease of reference:

{\footnotesize
\begin{tabular}[l]{c p{350pt}}
\\
$n$ & The number of PD requests; $|P| = |D| = n$;\\
$P$ & The set of pickup tasks, $P = \{1, 2, \dots, n\};$\\
$D$ & The set of delivery tasks, $D = \{1+n, 2+n, \dots, 2n\};$\\
$V'$ & The union of pickup and delivery tasks, $V'=P\cup D$ ;\\
$A'$ & The set of arcs $(i,j)$ with $i,j\in V'$;\\
$V$ & The full tasks including virtual start and end tasks, $V=\{0\}\cup V' \cup \{2n+1\}$;\\
$A$ & The set of arcs $(i,j)$ with $i,j\in V$;\\
$Q$ & The RGV's capacity in units of load (and every unit has the same weight \footnote{Generalization can be made to associate each unit of load with a different weight if such information is available.});\\
$q_i$ & The units of load associated with task $i$, satisfying $q_i>0$, $q_{i+n} <0$ and $q_i = -q_{i+n}$ for all $i\in P$. There are several types of containers with different sizes.  For example, the container in Figure \ref{fig: container_a} admits one unit of load, and so $q_i = 1, q_{i+n}=-1$; and the container in Figure \ref{fig: container_b} admits two units of load, and so $q_i=2, q_{i+n}=-2$;\\
$s_{i}$ & The operating duration of task $i\in V \backslash \{2n+1\}$, with $s_0 \triangleq 0$;\\
$[e_{i},l_{i}]$ & The delivery TW associated with each PD request $i \in P$;\\
$r_{ij}$  & The RGV travelling distance between tasks $i$ and $j$;\\
$t_{ij}$ & The RGV travelling time between tasks $i$ and $j$;\\
\\
\end{tabular}}

\begin{figure}
\centering
        \begin{subfigure}[b]{0.25\textwidth}
                \includegraphics[width=\textwidth]{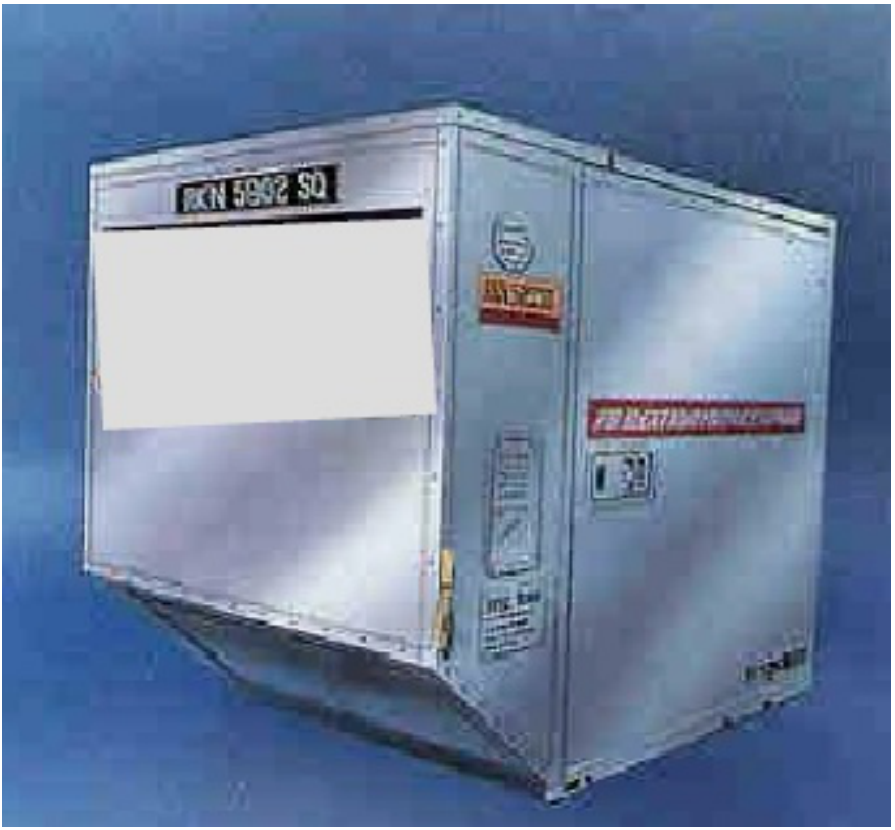}
                \caption{\footnotesize Container with unit load.}
                \label{fig: container_a}
        \end{subfigure}
        \begin{subfigure}[b]{0.407\textwidth}
                \includegraphics[width=\textwidth]{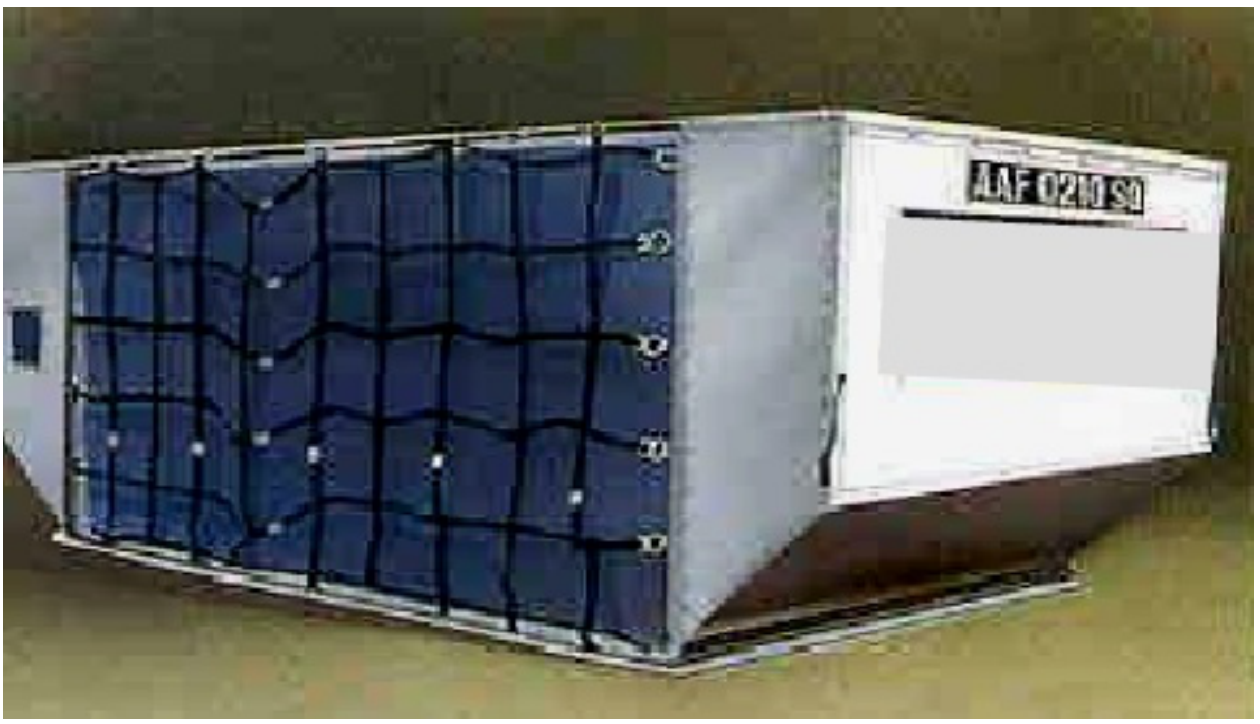}
                \caption{\footnotesize Containers with two units of load.}
                \label{fig: container_b}
        \end{subfigure}
\caption{Containers with different capacities.}
\label{fig: container}
\end{figure}

There are four groups of decision variables:

{\footnotesize
\begin{tabular}[l]{c p{350pt}}
\\
$x_{ij}$ & The binary variable to indicate  whether task $i$ is processed right before task $j$;\\
$y_{ij}^k$ & The binary variable to indicate whether the arc $(i,j)$ is traversed on the path from vertices $k$ to $k+n$, where $k \in P$;\\
$b_{i}$ & The time starting to handle task $i\in V$, with $b_{0} \triangleq 0$;\\
$w_{i}$ & The units of load on the RGV upon completion of task $i\in V$, with $w_{0} \triangleq 0$.\\
&
\end{tabular}}

The static optimal RGV routing problem can be formulated as a mixed integer program defined in (\ref{eq: PDP-TSLU})-(\ref{eq:c - b and w}). For convenience of presentation, we term it as the \emph{pickup and delivery problem with TSLU operations,} or \emph{PDP-TSLU} for short.

\begin{spacing}{1.0}
{\small
\begin{align}
\textbf{PDP-TSLU}: & \quad \min \sum_{i,j: (i,j)\in A}c_{ij}(w_{i})x_{ij}\label{eq: PDP-TSLU}\\
\text{subject to,} \quad \sum_{j:\,(i,j)\in A}x_{ij}=1 &  \quad \forall   i \in V\backslash \{2n+1\};\label{eq:c - outgoing degree}\\
\sum_{j:\,(j,i)\in A}x_{ji}=1 &  \quad \forall   i \in V\backslash \{0\};\label{eq:c - incoming degree}\\
b_{i+n}\ge b_{i}+s_{i}+t_{i,i+n} &  \quad \forall   i \in P;\label{eq:c - completion time}\\
b_{i \oplus 1}\ge b_{i}+s_{i} &  \quad \forall   b_{i \oplus 1}\ne \varnothing, i \in P;\label{eq:c - precedence in a pickup queue}\\
x_{ij}=1\Rightarrow\Biggl\{\begin{array}{l}
b_{j}\ge b_{i}+s_{i}+t_{ij}\\
w_{j}\ge w_{i}+q_{j}
\end{array} &  \quad \begin{array}{l}
\forall  (i,j)\in A,\,j \in V' \cup \{2n+1\},\\
\forall  (i,j)\in A,\,j \in V';
\end{array}\label{eq:c - time consistency}\\
e_{i}\le b_{i+n}+s_{i+n}\le l_{i} &  \quad \forall   i \in P;\label{eq:c - time window}\\
\max\{0, q_i\} \le w_{i} \le \min\{Q, Q+q_i\} &  \quad \forall   i \in V';\label{eq:c - load bound}\\
\sum_{j:\,(i,j)\in A'}y_{ij}^k-\sum_{j:\,(j,i)\in A'}y_{ji}^k=\begin{cases}
1 & \text{if }i=k\\
-1 & \text{if }i=k+n\\
0 & \text{otherwise}
\end{cases} &  \quad \forall   i \in V',\,k \in P;\label{eq:c - pickups to deliveries}\\
\eqref{eq:LIFO}-\eqref{eq:Deadlock-a}; \label{eq:c-TSLU} \\
y_{ij}^k\le x_{ij} &  \quad \forall  (i,j)\in A',\,k \in P; \label{eq:c - decision consistency}\\
y_{ij}^k\in\{0,\,1\} &  \quad \text{\ensuremath{\forall  }}(i,j)\in A',\,k \in P; \label{eq:c - y}\\
x_{ij}\in\{0,\,1\} &  \quad \text{\ensuremath{\forall  }}(i,j)\in A; \label{eq:c - x}\\
b_{i}\ge 0,w_{j}\ge 0 &  \quad \forall   i \in V'\cup \{2n+1\},\,j \in V'. \label{eq:c - b and w}
\end{align}
}{\small \par}
\end{spacing}

The objective function in (\ref{eq: PDP-TSLU}) measures the total energy consumption and is a \emph{bilinear} function in $w_{i}$ and $x_{ij}$, whose specific form is derived as follows. Each arc $(i,j)\in A$ is associated with an energy cost $c_{ij}(w_{i})$ (as explicitly depends on the load weight) and a travel
time $t_{ij}$. With the operating profile of the RGV shown in
Figure \ref{fig:RGV profile}, the arc travel time can be expressed as an explicit function of
the arc distance $r_{ij}$ as
\begin{equation}
t_{ij}=\begin{cases}
2\sqrt{\frac{r_{ij}}{a}} & \text{if }0\le r_{ij}\le2r_{1},\\
2t_{1}+\frac{r_{ij}-2r_{1}}{v_{c}} & \text{if }r_{ij}>2r_{1},
\end{cases}\label{eq: t_ij}
\end{equation}
where the parameters $a$, $v_{c}$, $t_{1}$ and $r_{1}$ are defined in Figure \ref{fig:RGV profile}. Let $w_{RGV}$ be the net weight of the empty RGV and $w_i$ denote the load weight upon leaving vertex $i$. The energy cost $c_{ij}(w_i)$ is calculated as the work done by the RGV for traversing the arc $(i,j)$. By Newton's law we have,
\begin{eqnarray}
c_{ij}(w_{i})= & \begin{cases}
\mu gr_{ij}(w_{RGV}+w_{i}) & \text{if}\,a\le\mbox{\ensuremath{\mu}}g,\\
ar_{ij}(w_{RGV}+w_{i}) & \text{if}\,a>\mu g\,\,\text{and }0\le r_{ij}\le2r_{1},\\
\left(2(a-\mu g)r_{1}+\mu gr_{ij}\right)(w_{RGV}+w_{i}) & \text{if}\,a>\mu g\,\,\text{and }r_{ij}>2r_{1},
\end{cases}\label{eq: c_ij}
\end{eqnarray}
where $\mu$ is the rolling friction coefficient while the RGV moves
on the track and $g$ is the gravitational acceleration, both of which
are given constants. Although $\mu$, $g$ and $r_{ij}$ are given constants, the energy cost $c_{ij}(w_i)$ are variables controllable by adjusting the task service order. This cost is generally increasing in $w_{i}$, in contrast to conventional routing cost which is merely determined by the arc distance $r_{ij}$. In the meanwhile, we set $c_{i,\,2n+1}=0$ for all $i\in D$, to account for the practice that the RGV stops at the last delivery.

\begin{figure}
\begin{centering}
\includegraphics[scale=0.7]{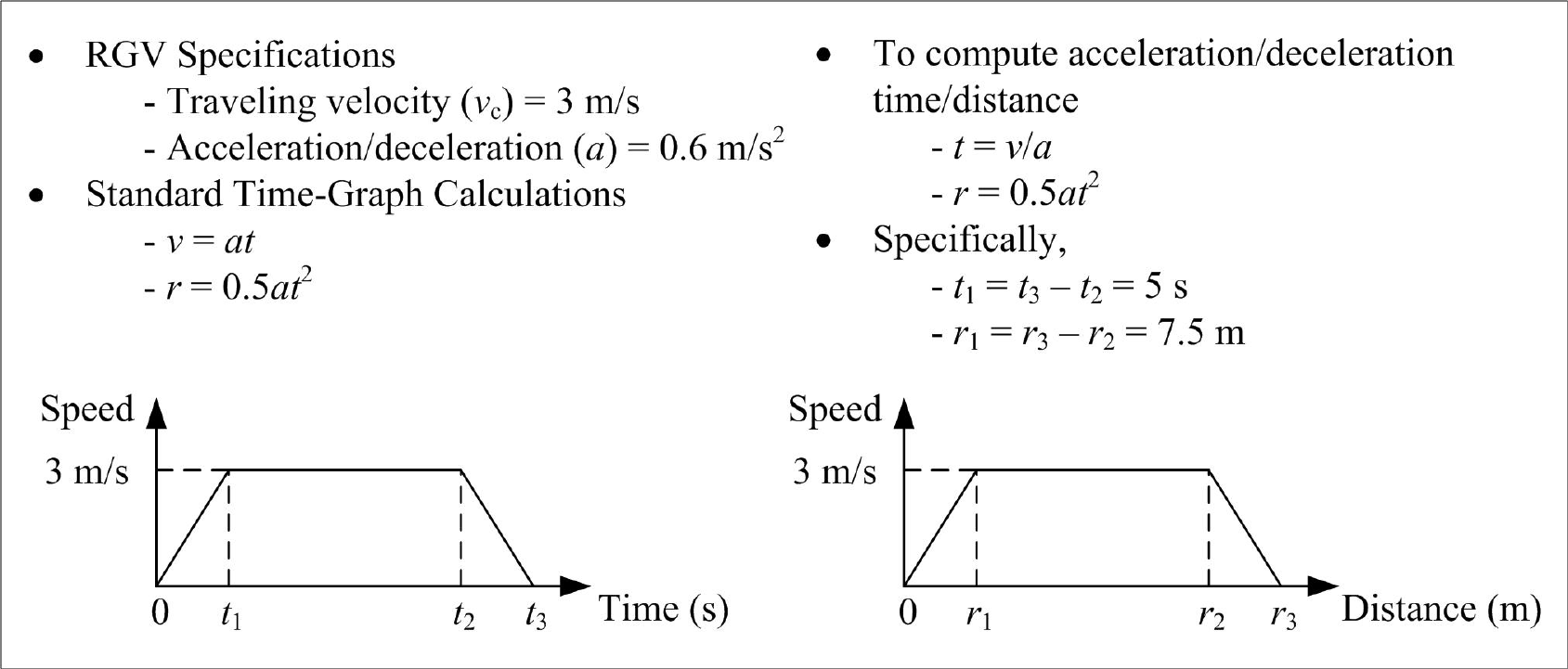}
\par\end{centering}
\caption{Operating profile of an RGV. \label{fig:RGV profile}}
\end{figure}

We briefly explain the constraints of PDP-TSLU. Constraints (\ref{eq:c - outgoing degree}) and (\ref{eq:c - incoming degree}) ensure that each task is served exactly once, with the service starting and ending at virtual tasks 0 and $2n+1$, respectively. Constraint (\ref{eq:c - completion time}) states that each pickup task should be processed before its paired delivery task, and constraint (\ref{eq:c - precedence in a pickup queue}) means that containers at the same work station are picked up in FIFO order. In (\ref{eq:c - precedence in a pickup queue}), the term $i\oplus 1$ represents the pickup task queued right behind task $i$ at the same station. For the example shown in Figure \ref{fig: toy example - scenario}, we have $1\oplus 1 = 2$ and $4\oplus 1 = 5$ .

The indicator constraint (\ref{eq:c - time consistency}) ensures the consistency of service time and load variables and also the route connectivity, which can be linearized using the big-$M$ method (refer to Appendix \ref{apdix: big-M formulation} for details). In  (\ref{eq:c - time consistency}), the load inequality constraints are valid alternatives of the corresponding equality constraints for the optimization considered, but this is true for the time constraints if only latest but earliest service constraints are active. Constraint (\ref{eq:c - time window}) imposes a TW constraint on completion time of each request. Constraint (\ref{eq:c - load bound}) enforces the RGV's capacity limit, which is often tighter than applying the obvious bound $0\le w_{i} \le Q$ for all $i\in V'$.

Constraint (\ref{eq:c - pickups to deliveries}) ensures a path from vertices $k$ to $k+n$ for any PD request $k$ that does not pass through the virtual start or end vertex. Constraint (\ref{eq:c-TSLU}) consists of unique conflict-free service constraints resulting from the TSLU operations as revealed in the previous subsection, and constraint (\ref{eq:c - decision consistency}) binds the overall routing decisions with the individual routing decisions of every PD request. Finally, constraints (\ref{eq:c - y})-(\ref{eq:c - b and w}) specify domains of the decision variables.

\begin{rem}
If the objective function is replaced with total travel distance or time, the PDP-TSLU reduces to TSPPDL confined on a linear track if there are only PD requests of Type-1 (or -2), and to TSPPDF confined on a linear track if there are only PD requests of Type-3 (or -4). In the two reduced cases, the conflict-free service constraints (\ref{eq:LIFO})-(\ref{eq:Deadlock-a}) collapse to mean the LIFO and the FIFO service order, respectively.
\end{rem}

\section{Solution approach} \label{sec: solution approach}

This section reformulates the bilinear PDP-TSLU into an MILP, and then reduces the arc set (and so the variables and constraints associated) based on delicate insights into the PDP-TSLU, and next introduces a couple of valid inequalities that exploit structural properties of the problem for solving it more efficiently.

\subsection{Reformulating PDP-TSLU as an MILP}

Let $\alpha_{ij}\triangleq ar_{ij}$ and $\beta_{ij}\triangleq 2(a-\mu g)r_{1}+\mu gr_{ij}$. The PDP-TSLU is bilinear in its present form because the objective function
has the following explicit form:
\begin{equation}
\sum_{(i,j)\in A}c_{ij}(w_{i})x_{ij}=\sum_{(i,j)\in A,r_{ij}\le2r_{1}}\alpha_{ij}\left(w_{RGV}+w_{i}\right)x_{ij}+\sum_{(i,j)\in A,r_{ij}>2r_{1}}\beta_{ij}\left(w_{RGV}+w_{i}\right)x_{ij},\label{eq: obj}
\end{equation}
which contains bilinear terms of $w_{i}x_{ij}$. This makes the problem
difficult to optimize directly. Fortunately, the objective function can be linearized by introducing
new decision variables and additional linear constraints.

Without loss of generality, let us assume $a>\mu g$ (which means
that braking is required for stopping a moving RGV at a deceleration
of $a$). In this case, the objective function of the PDP-TSLU is equivalent
to
\begin{eqnarray*}
\sum_{(i,j)\in A}c_{ij}(w_{i})x_{ij} & = & \sum_{i\in V\backslash\{2 n +1\}}z_{i},
\end{eqnarray*}
subject to, for all feasible index $j$,
\[
x_{ij}=1\Rightarrow z_{i}=\begin{cases}
\alpha_{ij}(w_{RGV}+w_{i}) & \forall  (i,j)\in A,\,r_{ij}\le2r_{1},\\
\beta_{ij}(w_{RGV}+w_{i}) & \forall  (i,j)\in A,\,r_{ij}>2r_{1}.
\end{cases}
\]
Note that it is unnecessary to introduce $z_{ij}$ for each
$x_{ij}$, because there always exists one and only one immediate successor $j$ to $i$ such that $x_{ij} = 1$ by constraint (\ref{eq:c - outgoing degree}).

The above indicator constraints can be handled directly by solvers embedded
in software like CPLEX \citep{studio2011v124}. Alternatively,
they can be reformulated into linear forms as
\begin{eqnarray*}
z_{i}\ge\alpha_{ij}(w_{RGV}+w_{i})-\gamma_{ij}(1-x_{ij}), &  & \forall  (i,j)\in A,\,r_{ij}\le2r_{1},\\
z_{i}\ge\beta_{ij}(w_{RGV}+w_{i})-\gamma_{ij}(1-x_{ij}), &  & \forall  (i,j)\in A,\,r_{ij}>2r_{1},
\end{eqnarray*}
where $\gamma_{ij}$ are upper bounds of $\alpha_{ij}(w_{RGV}+w_{i})$
if $r_{ij}\le2r_{1}$ and of $\beta_{ij}(w_{RGV}+w_{i})$ if $r_{ij}>2r_{1}$. It is feasible to set $\gamma_{ij}=\alpha_{ij}(w_{RGV}+\min\{Q,Q+q_{i}\})$
and $\gamma_{ij}=\beta_{ij}(w_{RGV}+\min\{Q,Q+q_{i}\})$ for the
two cases, respectively. Note that the inequality relaxations of the
indicator constraints do not change the optimal solution because the
objective function minimizes the sum of $z_{i}$ wtih unit
coefficients. Consequently, the PDP-TSLU becomes an MILP solvable
by standard solvers. The solution process can further be enhanced by exploiting structural properties underlying the problem, which are pursued in the next two subsections.

\subsection{Arc set reduction} \label{subsec: arc_set_reduction}

For practical reasons, not all arcs of the complete graph are feasible. By exploiting structural properties of the problem, the arc set $A$ in PDP-TSLU is reducible to the one defined in (\ref{eq:arc set}),
where the relation $i \lhd j $ (or $i \rhd j$) means that pickup task $i$ is queued in front of (or behind) pickup task $j$ at the same station and the term $i\oplus 1$ was defined when explaining \eqref{eq:c - precedence in a pickup queue}.

\begin{align}
A&=\Bigl\{(0,j): j\in P \; \mbox{such that}\; \{k: k\lhd j\} = \varnothing  \Bigr\}\nonumber\\
& \quad \cup\Bigl\{(i,\,2n+1): i\in D\; \mbox{such that}\; \{k: k \rhd i-n\} = \varnothing \Bigr\}\nonumber \\
& \quad \cup\Bigl\{ (i,j):\,i\in P,\,j\in P\cup D\backslash \cup_{l=1}^{4}S_{P,l}^{i} \backslash\cup_{l=1}^{3}S_{D,l}^{i}\Bigr\}\nonumber\\
& \quad \cup\Bigl\{ (i,j):\,i\in D,\,j\in P\cup D \backslash \cup_{l=5}^{9}S_{P,l}^{i} \backslash \cup_{l=4}^{8}S_{D,l}^{i} \Bigr\}\nonumber\\
& \quad \big\backslash\Bigl\{(i,\,j): i,j\in P\cup D\; \mbox{such that}\; i=j\;\mbox{or}\;i = j+n\Bigr\} ,\label{eq:arc set}
\end{align}
where
\begin{align*}
S_{P,\,1}^{i} & \triangleq \left\{ j\in P:j\lhd i\right\} ,\\
S_{P,\,2}^{i} & \triangleq \left\{ j\in P:j\rhd i\oplus1\right\} ,\\
S_{P,\,3}^{i} & \triangleq \left\{ j\in P:j\in P^{3}, \mbox{if}\;i\in P^{1}\cup P^{4}\right\}, \\
S_{P,\,4}^{i} & \triangleq \left\{ j \in P:\,j\in P^{4},\mbox{if}\;i\in P^{2}\cup P^{3}\right\} ,\\
S_{P,\,5}^{i} & \triangleq \left\{ j\in P:j \lhd i-n \right\} ,\\
S_{P,\,6}^{i} & \triangleq \bigl\{ j\in P: \{ k\in P^3: \,i-n\lhd k \lhd j \} \ne \varnothing, \,\mbox{if}\; i-n\in P^1 \bigr\}\\
S_{P,\,7}^{i} & \triangleq \bigl\{ j\in P: \{k\in P^4: \,i-n\lhd k \lhd j\}\ne \varnothing,\; \mbox{if}\;\,i-n\in P^2 \bigr\},\\
S_{P,\,8}^{i} & \triangleq \bigl\{ j\in P: \big|\,\{k\in P^3: i-n\lhd k\lhd j\}\,\big| \ge Q,\; \mbox{if}\;i-n \in P^3 \bigr\}\\
S_{P,\,9}^{i} & \triangleq \bigl\{ j\in P: \big|\,\{k\in P^4: i-n\lhd k\lhd j\}\,\big| \ge Q,\; \mbox{if}\;i-n \in P^4\bigr\},\\
S_{D,\,1}^{i} & \triangleq \left\{ j\in D:\,j - n \rhd i\right\} ,\\
S_{D,\,2}^{i} & \triangleq \left\{ j\in D:\,j - n\in P^{1}\cup P^{4}\backslash\{i\},\mbox{if}\;i\in P^{1}\cup P^{3}\right\} \\
S_{D,\,3}^{i} & \triangleq \left\{ j \in D:j - n \in P^{2}\cup P^{3}\backslash\{i\},\mbox{if}\;i\in P^{2}\cup P^{4}\right\} ,\\
S_{D,\,4}^{i} & \triangleq \left\{ j\in D:\,j - n \rhd i-n,\,\mbox{if}\;i-n\in P^{1}\cup P^{2}\right\} ,\\
S_{D,\,5}^{i} & \triangleq \left\{ j\in D:\,j - n \lhd i-n,\,\mbox{if}\;i-n\in P^{3}\cup P^{4}\right\} ,\\
S_{D,\,6}^{i} & \triangleq \left\{ j\in D:\,j - n \in S_{P,\,8}^{i} \cup S_{P,\,9}^{i}\right\} ,\\
S_{D,\,7}^{i} & \triangleq \left\{ j\in D:\,j - n\in P^{2}\cup P^{4}, \,\mbox{if}\;i-n\in P^{3}\right\} ,\\
S_{D,\,8}^{i} & \triangleq \left\{ j\in D:\,j - n\in P^{1}\cup P^{3}, \,\mbox{if}\;i-n\in P^{4}\right\} .
\end{align*}

Of the reduced arc set $A$, the first subset consists of arcs connecting
the start vertex with the head vertex of each pickup queue. The second
subset consists of arcs connecting each delivery vertex, whose paired pickup
vertex is the tail of a pickup queue, with the end vertex.
The third subset consists of arcs connecting a pickup vertex with
all other pickup vertices that are neither predecessors of the current
vertex ($ S_{P,\,1}^{i}$) nor successors other than the
nearest one in the same pickup queue ($ S_{P,\,2}^{i}$
) and meanwhile services of which avoid unloading deadlock (\ie, excluding $S_{P,\,3}^{i} \cup S_{P,\,4}^{i}$),
and with delivery vertices whose paired pickup vertices are not successors
of the current vertex ($ S_{D,\,1}^{i}$) and meanwhile
services of which avoid unloading deadlock (\ie, excluding $ S_{D,\,2}^{i}\cup S_{D,\,3}^{i}$).

The fourth subset of $A$ is most complicated. It includes arcs connecting a delivery vertex
with pickup vertices which are neither predecessors of the pickup
vertex currently in pair ($ S_{P,\,5}^{i}$) nor successors
that queue after the first successor of which the service leads to
unloading deadlock ($ S_{P,\,6}^{i} \cup S_{P,\,7}^{i}$), nor successors that
queue after the pickup vertex currently in pair for more than $Q$
request distance away while the requests are of Type-3 or Type-4 as the
same as the current request ($S_{P,\,8}^{i}\cup S_{P,\,9}^{i}$), and also includes arcs connecting a delivery vertex with
delivery vertices whose paired pickup vertices are neither successors
of the pickup vertex currently in pair if the current request is of
Type-1 or Type-2 ($ S_{D,\,4}^{i}$) nor predecessors or successors
that queue after the pickup vertex currently in pair for more than
$Q$ request distance away if the current request is of Type-3 or
Type-4 ($S_{D,\,5}^{i}\cup S_{D,\,6}^{i}$), and with delivery
vertices that are not on the same side of the current vertex while
the current request is of Type-3 or Type-4 (\ie, excluding $S_{D,\,7}^{i}\cup S_{D,\,8}^{i}$).

Overall the above reductions of the arc set owe to the precedence,
loading and capacity constraints inherent in the transport
service. To have a feel of the power of the above reductions, consider the
toy example shown in Figure \ref{fig: toy example - scenario}.
The arc set is reduced from an original size of 225 ($15^{2}$, for the complete
graph) to a size of 114 (for a much sparser graph), cutting almost
half of the full arcs. A direct consequence of such reductions is
that decision variables and constraints associated with the identified
infeasible arcs are eliminated before formulating and solving the
problem, contributing to a conciser model and hence higher computational
efficiency in general. It should be noted, however, that the elimination of infeasible
arcs is incomplete, which bears the necessity of using the
precedence, loading and capacity constraints in PDP-TSLU to achieve that.

\subsection{Valid inequalities for PDP-TSLU\label{sec: Valid-constraints}}

For a branch-and-cut algorithm to solve the PDP-TSLU, valid inequalities
exploiting structural properties of the problem can be useful to tighten
the lower bounds and avoid unnecessary branching which will consequently
expedite the process of searching for an optimal solution.

The PDP-TSLU is essentially a generic PDP subject to new loading/unloading constraints.
Therefore some valid inequalities known for a generic PDP apply to the PDP-TSLU.
New valid inequalities can also be derived from the unique conflict-free service constraints
in a way similar to those from LIFO or FIFO loading constraints as
presented in \citep{cordeau2010branch,cordeau2010branchFIFO}. Although
there are dozens of such valid inequalities, not all of them achieve
good trade-offs between performance benefit and implementation cost: According to the numerical
studies reported in \citep{cordeau2010branch,cordeau2010branchFIFO},
some kinds of valid inequalities require complicated implementations
(with respect to the demand for computer memory and computational
burden) but only slightly accelerate the solution process. Similar
kinds of valid inequalities are therefore ignored for the PDP-TSLU without
further validation.

\subsubsection{\textit{Valid inequalities inherited from generic PDP} }

Let $S$ be a subset of $P\cup D$. Define $x(S)=\sum_{i,j\in S}x_{ij}$, and let $\pi(S)=\{i\in P|i+n\in S\}$
and $\sigma(S)=\{i+n\in D|i\in S\cap P\}$ which denote the sets of predecessors
and successors of certain vertices in $S$, respectively, where $i+n$ is the delivery task paired with the pickup task $i \in P$. Define $\bar{S}=V\backslash S$. Then the following subtour elimination
constraints are well-known for the PDP (as a generic property of TSP
problems) \citep{cordeau2010branch,cordeau2010branchFIFO}:
\begin{equation}
x(S)\le|S|-1,\,\,\forall  \left|S\right|\ge2.\label{vc - subtour elimination}
\end{equation}
By taking account of precedence restrictions in PDP, these constraints
can further be strengthened as \citep{balas1995precedence,cordeau2010branch,cordeau2010branchFIFO}:
\begin{align}
x(S)+\sum_{i\in S}\sum_{j\in\bar{S}\cap\pi(S)}x_{ij}+\sum_{i\in S\cap\pi(S)}\sum_{j\in\bar{S}\backslash\pi(S)}x_{ij} & \le |S|-1,\label{vc - lifted subtour elimination1}\\
x(S)+\sum_{i\in\bar{S}\cap\sigma(S)}\sum_{j\in S}x_{ij}+\sum_{i\in\bar{S}\backslash\sigma(S)}\sum_{j\in S\cap\sigma(S)}x_{ij} & \le |S|-1.\label{vc - lifted subtour elimination2}
\end{align}

Another set of valid inequalities for the PDP are known as lifted $D_{k}^{+}$
and $D_{k}^{-}$ inequalities. The $D_{k}^{+}$ and $D_{k}^{-}$ inequalities
were firstly introduced by Grotschel and Padberg \citep{grotschel1985polyhedral}
and later strengthened by Cordeau \citep{cordeau2006branch}. The
lifted inequalities apply to any ordered set $S=\{i_{1},i_{2},\,\dots,i_{k}\}\subseteq V$
for $k\ge3$, and take the form of
\begin{align}
\sum_{h=1}^{k}x_{i_{h}i_{h+1}}+2\sum_{h=2}^{k-1}x_{i_{h}i_{1}}+\sum_{h=3}^{k-1}\sum_{l=2}^{h-1}x_{i_{h}i_{l}}+\sum_{j+n\in\bar{S}\cap\sigma(S)}x_{j+n,i_{1}} & \le k-1,\label{eq:vc - Dka}\\
\sum_{h=1}^{k}x_{i_{h}i_{h+1}}+2\sum_{h=3}^{k-1}x_{i_{1}i_{h}}+\sum_{h=4}^{k}\sum_{l=3}^{h-1}x_{i_{h}i_{l}}+\sum_{j\in\bar{S}\cap\pi(S)}x_{i_{1}j} & \le k-1,\label{eq:vc - Dkb}
\end{align}
where $i_{k+1}\triangleq i_{1}$.

\subsubsection{\textit{Valid inequalities derived from LIFO and CLO service restrictions}}

We derive new valid inequalities from the LIFO and CLO service constraints as unique to PDP-TSLU. Given $i\in P^{1},\,j\in P^{1}\backslash\{i\}$, if $x_{ij}=1$ is a feasible integer solution, then the pickup and delivery sequence must satisfy $0\prec i\prec j\prec j+n\prec i+n\prec 2n+1$, where the operator $\prec$ represents the precedence relationship. To enforce LIFO service order under TSLU operations, the predecessor of $i+n$ can only be a pickup task in $P^{2}\cup P^{4}$ or a delivery task of a request in $P\backslash(P^{4}\cup\{i\})$ (including $j$). Meanwhile, a valid successor of $j+n$ can only be a pickup task in $P\backslash \{P^3 \cup \{i, j\}\}$ or a delivery task of a request in $P^{2}\cup P^{3}\cup\{i\}$. Similar reasoning can be applied to determine the possible predecessors of $i+n$ and successors of $j+n$ for $i$ belonging to $P^{4}$ and for following CLO service order under TSLU operations. In summary, if arc $(i,j)$ is included in the routing path, then the set of possible predecessors of $i+n$ are given as
\begin{align*}
\mathcal{P}_{i+n}(i,j)=P^{2}\cup P^{4}\cup\{k+n:\,k\in P\backslash(P^{4}\cup\{i\})\}, & \quad  \forall   i\in P^{1}, j\in P^{1}\backslash\{i\},\\
\mathcal{P}_{i+n}(i,j)=\left(P^{2}\cup P^{4}\backslash\{i\}\right)\cup\{k+n:\,k\in P\backslash(P^{3}\cup\{i\})\}, &  \quad \forall   i\in P^{4}, j\in P^{1},\\
\mathcal{P}_{i+n}(i,j)=P^{1}\cup P^{3}\cup\{k+n:\,k\in P\backslash(P^{3}\cup\{i\})\}, & \quad  \forall   i\in P^{2}, j\in P^{2}\backslash\{i\},\\
\mathcal{P}_{i+n}(i,j)=\left(P^{1}\cup P^{3}\backslash\{i\}\right)\cup\{k+n:\,k\in P\backslash(P^{4}\cup\{i\})\}, & \quad  \forall   i\in P^{3}, j\in P^{2},
\end{align*}
and meanwhile the set of possible successors of $j+n$ are
\begin{align*}
\mathcal{S}_{j+n}(i,j)=\left(P\backslash(P^{3}\cup\{i,j\})\right)\cup\{k+n:\,k\in P^{2}\cup P^{3}\cup\{i\}\}, & \quad \forall   i\in P^{1}, j\in P^{1}\backslash\{i\},\\
\mathcal{S}_{j+n}(i,j)=\left(P\backslash(P^{3}\cup\{i,j\})\right)\cup\{k+n:\,k\in P\backslash(P^{3}\cup\{j\})\},  &  \quad\forall   i\in P^{4}, j\in P^{1},\\
\mathcal{S}_{j+n}(i,j)=\left(P\backslash(P^{4}\cup\{i,j\})\right)\cup\{k+n:\,k\in P^{1}\cup P^{4}\cup\{i\}\}, & \quad  \forall   i\in P^{2}, j\in P^{2}\backslash\{i\},\\
\mathcal{S}_{j+n}(i,j)=\left(P\backslash(P^{4}\cup\{i,j\})\right)\cup\{k+n:\,k\in P\backslash(P^{4}\cup\{j\})\},  & \quad \forall   i\in P^{3}, j\in P^{2}.
\end{align*}

On the other hand, given $i\in P^{1},\,j\in P^{1}\backslash\{i\}$, if $x_{j+n, i+n}=1$ is a feasible solution, then the pickup and delivery sequence again must follow the aforementioned precedence order. To enforce the LIFO service order under TSLU operations, the predecessor of $j$ can only be a pickup task in $\{i\}\cup P^{2}\cup P^{4}$ or a delivery task of a request in $P\backslash(P^{4}\cup\{i,j\})$. Meanwhile, a valid successor of $i$ can only be a pickup task in $P\backslash(P^{3}\cup\{i\})$ or a delivery task of a request in $P^{2}\cup P^{3}$. Similar reasoning can be applied to determine the possible predecessors of $j$ and successors of $i$ for $i$ belonging to $P^{4}$ and for satisfying the CLO service order under TSLU operations. In summary, if arc $(j+n,i+n)$ is included in the routing path, then the set of possible predecessors of $j$  are given as
\begin{align*}
\mathcal{P}{}_{j}(j+n,i+n)=\{i\}\cup P^{2}\cup P^{4}\cup\{k+n:\,k\in P\backslash(P^{4}\cup\{i,j\})\}, &  \quad \forall   i\in P^{1}, j\in P^{1}\backslash\{i\},\\
\mathcal{P}{}_{j}(j+n,i+n)=\{0\}\cup P^{2}\cup P^{4}\cup\{k+n:\,P\backslash(P^{3}\cup\{i,j\})\}, &  \quad \forall   i\in P^{4}, j\in P^{1},\\
\mathcal{P}{}_{j}(j+n,i+n)=\{i\}\cup P^{1}\cup P^{3}\cup\{k+n:\,k\in P\backslash(P^{3}\cup\{i,j\})\}, &  \quad \forall   i\in P^{2}, j\in P^{2}\backslash\{i\},\\
\mathcal{P}{}_{j}(j+n,i+n)=\{0\}\cup P^{1}\cup P^{3}\cup\{k+n:\,P\backslash(P^{4}\cup\{i,j\})\}, &  \quad \forall   i\in P^{3}, j\in P^{2},
\end{align*}
and meanwhile the set of possible successors of $i$ are
\begin{align*}
\mathcal{S}{}_{i}(j+n,i+n)=\left(P\backslash(P^{3}\cup\{i\})\right)\cup\{k+n:\,k\in P^{2}\cup P^{3}\}, &  \quad \forall   i\in P^{1}, j\in P^{1}\backslash\{i\},\\
\mathcal{S}{}_{i}(j+n,i+n)=\left(P\backslash(P^{3}\cup\{i\})\right)\cup\{k+n:\,k\in P^{1}\cup(P^{4}\backslash\{i\})\}, &  \quad \forall   i\in P^{4}, j\in P^{1},\\
\mathcal{S}{}_{i}(j+n,i+n)=\left(P\backslash(P^{4}\cup\{i\})\right)\cup\{k+n:\,k\in P^{1}\cup P^{4}\}, &  \quad \forall   i\in P^{2}, j\in P^{2}\backslash\{i\},\\
\mathcal{S}{}_{i}(j+n,i+n)=\left(P\backslash(P^{4}\cup\{i\})\right)\cup\{k+n:\,k\in P^{2}\cup(P^{3}\backslash\{i\})\}, &  \quad \forall   i\in P^{3}, j\in P^{2}.
\end{align*}
In the second and fourth cases of predecessors set of $j$, the vertex $0$ instead of $i$ is included in the set because $j+n\prec i+n$ does not imply $i\prec j$ in those two cases.

With the above insights, four groups of valid inequalities can be obtained for the
PDP-TSLU, each of which indicates a group of arcs that are incompatible with the arc $(i, j)$ or $(j+n, i+n)$ as included in a valid routing path.
\begin{prop}
For each of the vertex pairs $i\in P^{1}\cup P^{4},\,j\in P^{1}\backslash\{i\}$,
or $i\in P^{2}\cup P^{3},\,j\in P^{2}\backslash\{i\}$, the following
inequalities hold for the PDP-TSLU:
\begin{align}
x_{ij}+\sum_{l\notin\mathcal{P}_{i+n}(i,j)}x_{l,i+n} & \le 1,\label{eq: vc - LIFO1a}\\
x_{ij}+\sum_{l\notin\mathcal{S}_{j+n}(i,j)}x_{j+n,l} & \le 1,\label{eq: vc - LIFO1b}\\
x_{j+n,i+n}+\sum_{l\notin\mathcal{P}_{j}(j+n,i+n)}x_{lj} & \le 1,\label{eq: vc - LIFO2a}\\
x_{j+n,i+n}+\sum_{l\notin\mathcal{S}_{i}(j+n,i+n)}x_{il} & \le 1,\label{eq: vc - LIFO2b}
\end{align}
where the variables are zero if their corresponding arcs are not defined
in $A$.
\end{prop}
The next inequality comes from the fact that whenever an arc $(i,i+n)$
for any $i\in P^{3}\cup P^{4}$ is used, it follows from the CLO loading
restriction that the vehicle must arrive empty at $i$ and leaves
empty from $i+n$.
\begin{prop}
For each of the vertex pairs $i\in P^{4},\,j\in P^{1}$, or $i\in P^{3},\,j\in P^{2}$,
the following inequality holds for the PDP-TSLU:
\begin{equation}
x_{ij}+x_{ji}+x_{i,i+n}+x_{i+n,j+n}\le1,\label{eq:vc - LIFO3a}
\end{equation}
where the variables are zero if their corresponding arcs are not defined
in $A$.
\end{prop}
The following inequality merges the two valid inequalities $x_{ij}+x_{i+n,j+n}+x_{j+n,i}\le1$
and $x_{ij}+x_{i+n,j+n}+x_{i+n,j}\le1$, for each pair of $i$
and $j$ in proper sets.
\begin{prop}
For each of the vertex pairs $i\in P^{1}\cup P^{4},\,j\in P^{1}\backslash\{i\}$,
or $i\in P^{2}\cup P^{3},\,j\in P^{2}\backslash\{i\}$, the following
inequality holds for the PDP-TSLU:
\begin{equation}
x_{ij}+x_{i+n,j+n}+x_{j+n,i}+x_{i+n,j}\le1,\label{eq:vc - LIFO3b}
\end{equation}
where the variables are zero if their corresponding arcs are not defined
in A.
\end{prop}

\subsubsection{\textit{Valid inequalities derived from FIFO and CLO service restrictions}}

This subsection derives new valid inequalities pertinent to the FIFO and CLO service constraints of the PDP-TSLU. Given $i\in P^{4},\,j\in P^{4}\backslash \{i\}$, if $x_{ij}=1$ is a feasible solution, then the pickup and delivery sequence must satisfy $0 \prec i \prec j \prec i+n\prec j+n\prec 2n+1$. To enforce the FIFO service order under TSLU operations, the successor of $i+n$ can only be a pickup task in $P\backslash(P^{3}\cup\{i,j\})$ or a delivery task of a request in $P^{2}\cup\{j\}$. Meanwhile, a valid predecessor of $j+n$ can only be a pickup task in $P^{2}\cup(P^{4}\backslash\{i,j\})$ or a delivery task of a request in $P^{1}\cup P^{2}\cup\{i\}$. Similar reasoning can be applied to determine the possible successors of $i+n$ and predecessors of $j+n$ for $i$ belonging to $P^{1}$ and for satisfying CLO service order under TSLU operations. In summary, if arc $(i,j)$ is included in the routing path, then the set of possible successors of $i+n$ are obtained as
\begin{align*}
\mathcal{S}_{i+n}(i,j)=\left(P\backslash(P^{3}\cup\{i,j\})\right)\cup\left\{ k+n:\,k\in P\backslash(P^{3}\cup\{i\})\right\} , &  \quad \forall   i\in P^{1}, j\in P^{4},\\
\mathcal{S}_{i+n}(i,j)=\left(P\backslash(P^{3}\cup\{i,j\})\right)\cup\left\{ k+n:\,k\in P^{2}\cup\{j\}\right\} , &  \quad \forall   i\in P^{4}, j\in P^{4}\backslash\{i\},\\
\mathcal{S}_{i+n}(i,j)=\left(P\backslash(P^{4}\cup\{i,j\})\right)\cup\left\{ k+n:\,k\in P\backslash(P^{4}\cup\{i\})\right\} , &  \quad \forall   i\in P^{2}, j\in P^{3},\\
\mathcal{S}_{i+n}(i,j)=\left(P\backslash(P^{4}\cup\{i,j\})\right)\cup\left\{ k+n:\,k\in P^{1}\cup\{j\}\right\} , &  \quad \forall   i\in P^{3}, j\in P^{3}\backslash\{i\},
\end{align*}
and meanwhile the set of possible predecessors of $j+n$ are
\begin{align*}
\mathcal{P}_{j+n}(i,j)=P^{2}\cup(P^{4}\backslash\{j\})\cup\left\{ k+n:\,k\in P\backslash(P^{3}\cup\{j\})\right\} , &  \quad \forall   i\in P^{1}, j\in P^{4},\\
\mathcal{P}_{j+n}(i,j)=P^{2}\cup(P^{4}\backslash\{i,j\})\cup\left\{ k+n:\,k\in P^{1}\cup P^{2}\cup\{i\}\right\} , &  \quad \forall   i\in P^{4}, j\in P^{4}\backslash\{i\},\\
\mathcal{P}_{j+n}(i,j)=P^{1}\cup(P^{3}\backslash\{j\})\cup\left\{ k+n:\,k\in P\backslash(P^{4}\cup\{j\})\right\} , &  \quad \forall   i\in P^{2}, j\in P^{3},\\
\mathcal{P}_{j+n}(i,j)=P^{1}\cup(P^{3}\backslash\{i,j\})\cup\left\{ k+n:\,k\in P^{1}\cup P^{2}\cup\{i\}\right\} , &  \quad \forall   i\in P^{3}, j\in P^{3}\backslash\{i\}.
\end{align*}

Given $i\in P^{4},\,j\in P^{4}$, if $x_{i+n, j+n}=1$ is a feasible solution, then to comply with the FIFO service order under TSLU operations, the successor of $i$ can only be a pickup vertex in $\{j\}\cup P^{1}\cup P^{2}$ or a delivery vertex of a request in $P^{1}\cup(P^{4}\backslash\{j\})$. Meanwhile, a valid predecessor of $j$ can only be a pickup vertex in $\{i\}\cup P^{1}$ or a delivery vertex of a request in $P\backslash(P^{3}\cup\{i,j\})$. Similar reasoning can be applied to determine the possible successors of $i$ and predecessors of $j$ for $i$ belonging to $P^1$ and for satisfying CLO service order under TSLU operations. In summary, if arc $(i+n, j+n)$ is included in the routing path, then the set of possible successors of $i$ are obtained as
\begin{align*}
\mathcal{S}_{i}(i+n,j+n)=\{j\}\cup(P^{1}\backslash\{i\})\cup P^{2}\cup\left\{ k+n:\,k\in P^{2}\cup P^{3}\cup\{i\}\right\} , &  \quad \forall   i\in P^{1}, j\in P^{4},\\
\mathcal{S}_{i}(i+n,j+n)=\{j\}\cup P^{1}\cup P^{2}\cup\left\{ k+n:\,k\in P^{1}\cup(P^{4}\backslash\{j\})\right\} , &  \quad \forall   i\in P^{4}, j\in P^{4}\backslash\{i\},\\
\mathcal{S}_{i}(i+n,j+n)=\{j\}\cup P^{1}\cup(P^{2}\backslash\{i\})\cup\left\{ k+n:\,k\in P^{1}\cup P^{4}\cup\{i\}\right\} , &  \quad \forall   i\in P^{2}, j\in P^{3},\\
\mathcal{S}_{i}(i+n,j+n)=\{j\}\cup P^{1}\cup P^{2}\cup\left\{ k+n:\,k\in P^{2}\cup(P^{3}\backslash\{j\})\right\} , &  \quad \forall   i\in P^{3}, j\in P^{3}\backslash\{i\},
\end{align*}
and meanwhile the set of possible predecessors of $j$ are
\begin{align*}
\mathcal{P}_{j}(i+n,j+n)=\{0\}\cup P^{1}\cup(P^{4}\backslash\{j\})\cup\left\{ k+n:\,k\in P\backslash\{i,j\}\right\} , &  \quad \forall   i\in P^{1}, j\in P^{4},\\
\mathcal{P}_{j}(i+n,j+n)=\{i\}\cup P^{1}\cup\left\{ k+n:\,k\in P\backslash(P^{3}\cup\{i,j\})\right\} , &  \quad \forall   i\in P^{4}, j\in P^{4}\backslash\{i\},\\
\mathcal{P}_{j}(i+n,j+n)=\{0\}\cup P^{2}\cup(P^{3}\backslash\{j\})\cup\left\{ k+n:\,k\in P\backslash\{i,j\}\right\} , &  \quad \forall   i\in P^{2}, j\in P^{3},\\
\mathcal{P}_{j}(i+n,j+n)=\{i\}\cup P^{2}\cup\left\{ k+n:\,k\in P\backslash(P^{4}\cup\{i,j\})\right\} , &  \quad \forall   i\in P^{3}, j\in P^{3}\backslash\{i\}.
\end{align*}

With the above insights, four groups of valid inequalities in analog to those in
(\ref{eq: vc - LIFO1a})-(\ref{eq: vc - LIFO2b}) can be obtained to invalidate infeasible predecessors and successors to a given task:
\begin{prop}
For each of the vertex pairs $i\in P^{1}\cup P^{4},\,j\in P^{4}\backslash\{i\}$,
or $i\in P^{2}\cup P^{3},\,j\in P^{3}\backslash\{i\}$, the following
inequalities hold for the PDP-TSLU:
\begin{align}
x_{ij}+\sum_{l\notin\mathcal{S}_{i+n}(i,j)}x_{i+n,l} & \le 1,\label{vc - FIFO1a}\\
x_{ij}+\sum_{l\notin\mathcal{P}_{j+n}(i,j)}x_{l,j+n} & \le 1,\label{vc - FIFO1b}\\
x_{i+n,j+n}+\sum_{l\notin\mathcal{S}_{i}(i+n,j+n)}x_{il} & \le 1,\label{vc - FIFO2a}\\
x_{i+n,j+n}+\sum_{l\notin\mathcal{P}_{j}(i+n,j+n)}x_{lj} & \le 1.\label{vc - FIFO2b}
\end{align}

\end{prop}
The next inequality is adapted from the one proposed for TSPPDF in \citep{cordeau2010branchFIFO}.
It comes from the fact that whenever an arc $(i,i+n)$ for any
$i\in P^{3}\cup P^{4}$ is traversed, it follows from the FIFO service
restriction that the vehicle must arrive empty at $i$ and leaves
empty from $i+n$.
\begin{prop}
For each of the vertex pairs $i\in P^{4},\,j\in P^{4}\backslash\{i\}$,
or $i\in P^{3},\,j\in P^{3}\backslash\{i\}$, the following inequality
holds for the PDP-TSLU:
\begin{equation}
x_{ji}+x_{i,i+n}+x_{i+n,j+n}\le1,\label{eq:vc - FIFO3a}
\end{equation}
where the variables are zero if their corresponding arcs are not defined
in A.
\end{prop}
The following inequality is again adapted from \citep{cordeau2010branchFIFO},
which merges the two inequalities $x_{ij}+x_{j+n,i+n}+x_{j+n,i}\le1$
and $x_{ij}+x_{j+n,i+n}+x_{i+n,j}\le1$ , for each pair $i$ and
$j$ in proper request sets.
\begin{prop}
For each of the vertex pairs $i\in P^{1}\cup P^{4},\,j\in P^{4}\backslash\{i\}$,
or $i\in P^{2}\cup P^{3},\,j\in P^{3}\backslash\{i\}$, the following
inequality holds for the PDP-TSLU:
\begin{equation}
x_{ij}+x_{j+n,i+n}+x_{j+n,i}+x_{i+n,j}\le1,\label{eq:vc - FIFO3b}
\end{equation}
where the variables are zero if the corresponding arcs are not defined
in A.
\end{prop}
The valid inequalities (\ref{vc - lifted subtour elimination1})-(\ref{eq:vc - FIFO3b})
will be used to enhance the branch-and-cut algorithm embedded in
CPLEX for solving the PDP-TSLU, and their usefulness will be evaluated
via numerical experiments.

\begin{rem}
The above valid inequalities are not exclusive to the PDP-TSLU and other derivations are possible. For example, valid inequalities can also be derived from the CFI service restrictions in a similar way. Our experiments have shown that they do not contribute much to the computational efficiency, and hence are ignored in our studies. For another example, valid inequalities can be derived from the geometric
restriction of a linear track, known as the ``edge degree balance''
property \citep{atallah1987efficient}. The property says that, for
an RGV (which starts and ends at the same position) traversing any
track segment, the number of times it moves towards the right must
be equal to that towards the left. The property remains true for the
PDP-TSLU if a virtual arc is introduced to connect the end vertex
with the start vertex. Simulations have shown
that such valid inequalities hardly accelerate the solution process of PDP-TSLU and hence are ignored also.
\end{rem}

\section{Rolling-horizon approach for handling dynamic PDP-TSLU} \label{sec: rolling-horizon-appro}

In practice PD requests arrive stochastically and the PDP-TSLU becomes dynamic. Ideally the dynamic PDP-TSLU should be solved to optimality at each decision point as a static alternative which sequences all PD tasks up to the decision point. In reality, however, a computer has limited computational capacity and a suboptimal solution which considers a limited number of PD requests has to be
explored at each decision point. This motivates us to adopt a rolling-horizon approach for handling the dynamic PDP-TSLU. To that end, we need to treat PDP-TSLU with nonzero
initial load first.

\subsection{Handling nonzero initial conditions}

At the time of recomputing a routing solution, the RGV may contain
containers which have not been delivered, \ie, $w_{0}\ge1$. In this case,
the re-optimization is subject to a nonzero initial condition. The
initial load on the RGV may consist of containers that are destined to
stations on the same or different sides of the track, as depicted in
Figure \ref{fig:Initial requests}, where the impossible types
of initial requests are avoided by the previous routing solution.

\begin{figure}
\begin{centering}
\includegraphics[scale=0.7]{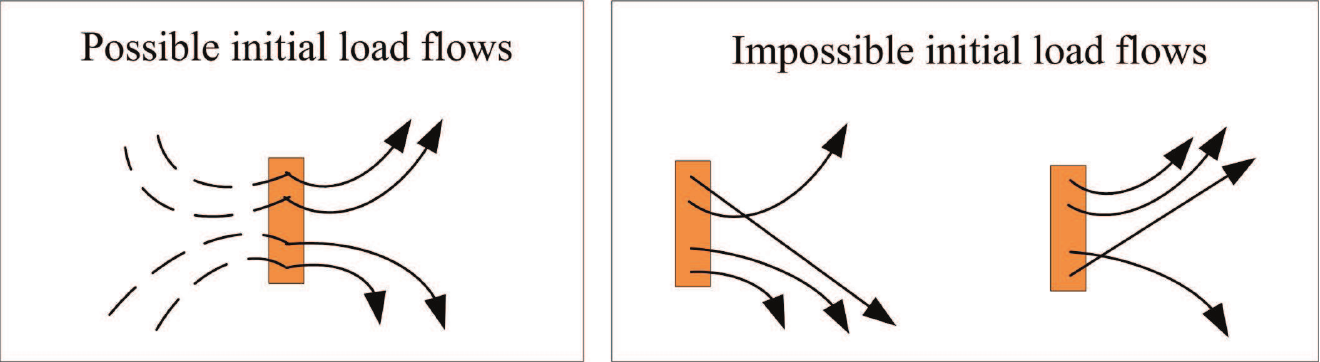}
\par\end{centering}
\caption{Initial delivery requests and their completion as virtual PD requests.
\label{fig:Initial requests}}
\end{figure}

To recompute a routing decision, extra constraints are required for handling new requests without causing conflict with the existing load. Characterizing these constraints
and incorporating them into the present routing model are thus necessary.
The renewed problem can be reformulated as a standard one with zero initial load and partially determined arcs, as explained below.

Let there be $n_{0}$ delivery requests at the time when a
renewal of routing is triggered, of which $n_{0,1}$ are destined
to north stations and $n_{0,2}$ (equal to $n_{0}-n_{0,1}$)
to south stations. In the graph model, it is feasible
to represent the $n_{0,1}$ transport requests by $n_{0,1}$ virtual PD requests of Type-1, each associated with a travel distance equal to the current distance of the RGV to the
destined station (because visiting a virtual pickup vertex does not introduce a routing cost).
Similarly, the rest $n_{0,2}$ transport requests can be represented
by $n_{0,2}$ virtual PD requests of
Type-2. See Figure \ref{fig:Initial requests} for a depiction of
the virtual PD requests, as indicated by arrows comprised of dashed
and solid lines. In the meantime, the current position of the RGV is represented
by a virtual start vertex as directed to one of the virtual
pickup vertices with zero routing cost. The order of visiting the virtual
start and pickup vertices is specified in such a way that yields the initial load setting.

The routing model is then augmented to include the initial delivery
vertices, the virtual pickup and start vertices, and the
arcs associated. Denote the set of virtual pickup vertices of
Type-1 as $P_{0}^{1} \triangleq \{1, 2, \dots, n_{0,1}\}$, and those of Type-2 as $P_{0}^{2} \triangleq \{n_{0,1}+1,n_{0,1}+2, \dots, n_0\}$, and let $P_{0}\triangleq P_{0}^{1} \cup P_{0}^{2}$, as a subset of the augmented pickup vertex set $P$. Without loss of generality, we let the RGV start at the virtual start vertex 0 and visit the virtual pickup vertices in the order of $1, 2, \dots, n_0$,
resulting in the initial load setting. The arcs
associated with these virtual vertices ($\{0\}\cup P_{0}$)
are thus known a priori and the related arc set reduces to
\begin{align*}
A_{0} & = \left\{ (k, k+1):\,k=0, 1, \dots, n_0-1 \right\} \\
&  \quad \cup\left\{ (n_0, j):\, j \in P^{1} \cup P^{2} \cup \left\{ n_{0,1}+h, n_{0,2}+h \right\}\right\}\\
 &  \quad \cup\left\{ (n_0, j): j \in P^{3}, \text{ if } n_{0,1}=0 \right\}\\
  &  \quad \cup\left\{ (n_0, j): j \in P^{4}, \text{ if } n_{0,2}=0 \right\} ,
\end{align*}
where $h$ is a given size of the rolling horizon, satisfying $h \ge n_0$. All arcs defined in $A$ directing to $\{0\}\cup P_{0}$
become null since all these vertices have already been visited. For each arc $(i,j)$
in the first subset of $A_{0}$ above, the arc distance $r_{ij}$ is equal to zero; and for each arc $(i,j)$
in the rest subsets of $A_{0}$ above, $r_{ij}$ is equal to the distance
of the current position of the RGV to vertex $j$.

Consequently, the RGV routing problem with nonzero
initial load is transformed into a PDP-TSLU with partially determined arcs, in which an empty RGV starts at a virtual
vertex 0, visits the virtual pickup vertices in a predefined order,
and then the rest pickup and delivery vertices within the rolling
horizon by following an order determined by solving the
PDP-TSLU.

\subsection{The rolling-horizon approach}

As PD requests arrive stochastically in a complex manner, we adopt a rolling-horizon approach
to handling the dynamic PDP-TSLU. To make the approach work in real time,
we need to address two issues: One is to select an appropriate
rolling-horizon size and the other is about online implementation
of the routing decisions. In principle, the rolling horizon should be as long
as possible in order to avoid myopic decisions, but in the meantime it must be short enough to afford real-time decisions. The size of the horizon is thus limited by the computational time required
for solving a static PDP-TSLU, and the specific size needs to be determined
via simulations. Regarding the implementation of rolling decisions,
two situations have to be handled appropriately: One is when re-optimization
is demanded while the previous optimization is still in progress; and
the other is when a new decision is available while the RGV is on
the way of executing the previous decision. With these situations
in mind, the rolling-horizon approach is designed to work as follows.

\emph{Rolling-horizon approach}: Whenever a new PD request is issued or there are un-sequenced PD
tasks in the system, the approach recomputes a solution of sequencing of
PD requests by solving the (augmented) PDP-TSLU with a given horizon size (if there were so many requests),
say, $h$. If there are more than $h$ PD requests
to sequence, $h$ of them are selected based on the ``Earliest Due
First'' rule while satisfying the FIFO service order at each station.
In this way, the RGV keeps serving PD tasks as sequenced by
the previous decision until a new decision comes out. In addition,
executing the first task of each decision is made compulsory and a new
decision is adopted only after the RGV completes a current task of the
previous decision.

Following this approach, an RGV is able to handle PD requests continuously
without obvious disruption although the decision is renewed over
time. This is true as long as the computation of a new decision
can be finished before the RGV completes all requests in the previous horizon.

\section{Computational studies} \label{sec: computational studies}

A toy example is first given to illustrate the modeling and solution
of a static PDP-TSLU. Then computational results of static PDP-TSLUs with
randomly generated requests of various sizes are presented to evaluate
the usefulness of the valid inequalities for solving the PDP-TSLUs. Based
on the results, an appropriate horizon size is selected for
implementing the rolling-horizon approach. Then the rolling-horizon
approach is applied to handle dynamic PDP-TSLUs of random instances, and the results are compared with those obtained with a typical rule-based method.

The basic scenario settings of the static and dynamic PDP-TSLUs are as follows. In each instance, the PD requests are randomly generated within
a work area. Specifically, for each static instance a queue
of PD requests at any station on both sides of the track
are generated with a size following a uniform distribution within
$[0,a]$, where $a$ is a given integer which is different for
instances of different sizes. For each dynamic instance, the arrival
time of a new PD request follows a Poisson distribution and its location
follows a uniform distribution within the range of stations. The distance
between two neighboring stations is treated as a unit distance, and a
unit service time is assumed for each PD request as equally divided
between the pickup and the delivery operations. In our studies, an
RGV is assumed to have a constant mass equal to two units of load.
The travel distance of an RGV is defined as the total distance it
travels before the last delivery, and the energy cost of
an RGV is defined as per (\ref{eq: obj}). The friction
coefficient of the track is set as 0.05 and the gravitational acceleration as 9.8
$m/s^{2}$. The problems are coded in C++ and solved via
IBM ILOG CPLEX 12.4 \citep{studio2011v124} which runs on a PC with
Intel Core2 Duo CPU @ 2.66 GHz and 2 GB of RAM.

\subsection{A toy example of static PDP-TSLU}

Consider seven PD requests distributed in nine pairs of face-to-face
stations as shown in Figure \ref{fig: toy example - scenario}, whose graph representation is shown in Figure \ref{fig: toy_example}.
In the representation, pickup vertices sit in FIFO queues while
delivery vertices do not follow any precedence order and are represented
to sit side by side of each other. The RGV starts from location 3.

\begin{figure}
\centering
        \begin{subfigure}[b]{0.4\textwidth}
                \includegraphics[width=\textwidth]{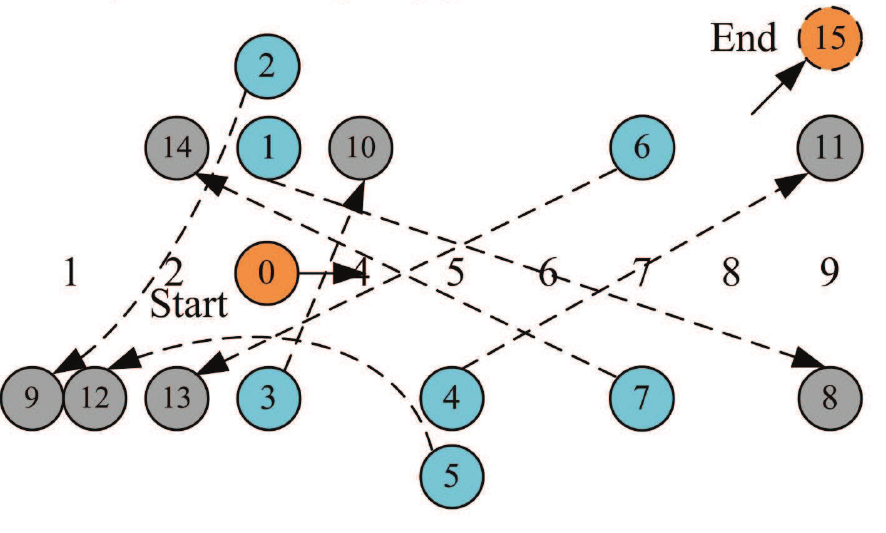}
                \caption{\footnotesize RGV starting with empty load.}
                \label{fig: toy_example_a}
        \end{subfigure}
        \begin{subfigure}[b]{0.4\textwidth}
                \includegraphics[width=\textwidth]{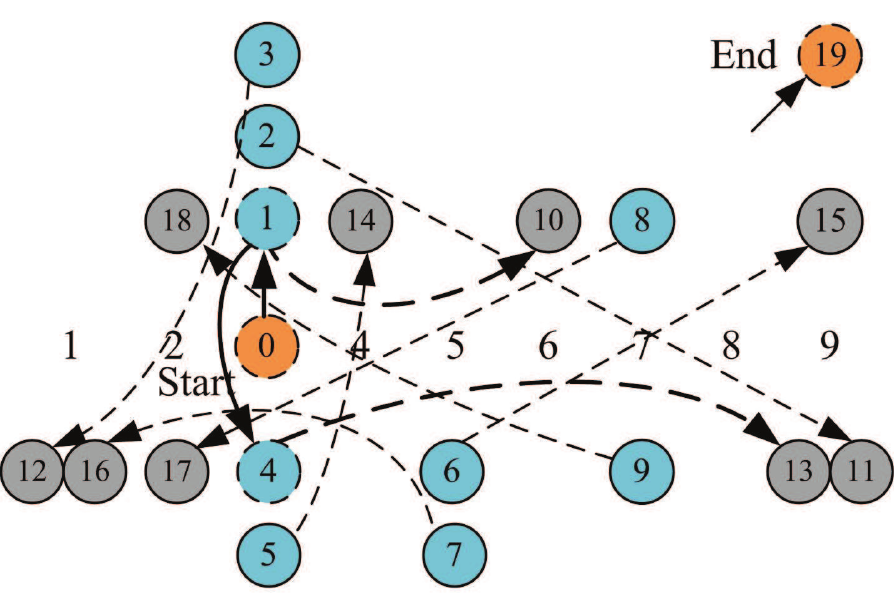}
                \caption{\footnotesize RGV starting with two units of load.}
                \label{fig: toy_example_b}
        \end{subfigure}
\caption{Graph representations of scenarios with empty and nonempty initial
load.}
\label{fig: toy_example}
\end{figure}

Two scenarios are investigated: (a) the RGV starts with empty load, and
(b) the RGV starts with two units of load on board. The RGV is simulated with different capacities. The numerical results of scenario (a) are summarized in part
(a) of Table \ref{tab:Toy example}. We observe that
the energy cost drops considerably once the capacity increases from
1 to 2, but it ceases dropping as the RGV's capacity becomes larger. This is because the service restrictions and the weight-dependent
energy consumption are to prevent full loading and hence limit the benefit
of having a larger capacity. The results also indicate that there can be
multiple solutions for achieving the same energy cost. In
comparison, minimizing travel distance results in higher energy costs for all RGV capacities tested, whose values are 44.59, 52.43, 52.43, 48.51 and 52.43 for the capacity equal to $2, 3, \dots, 6$, respectively (while achieving travel distances of 30, 29, 29, 29 and 29, correspondingly).

In scenario (b), the RGV starts with two units of load
on board, one to be delivered to the north station at location 6
and the other to the south station at location 9. The other PD requests
keep the same as in scenario (a). By introducing a virtual start vertex
(0) and two virtual pickup vertices (1 and 4), the graph model is
augmented as in Figure \ref{fig: toy_example_b}. By solving the
augmented PDP-TSLU with zero initial load and partially determined
arcs under five different RGV capacities, it yields the computational results summarized in part (b) of Table \ref{tab:Toy example}.
The routing solutions turn out to be the same for the five choices
of capacity, resulting in a total energy cost of 55.86. In comparison, minimizing travel distance gives total
energy costs of 56.84, 62.72, 66.64, 66.64 and 64.68 for the five capacities, respectively
(all achieving a travel distance of 36), which are higher
than the aforementioned minimum counterparts obtained by solving the PDP-TSLUs.

\begin{table*}
{\small \caption{{\footnotesize Computational results of the toy example.\label{tab:Toy example}}}
}{\small \par}

\centering{}%
{\scriptsize
\begin{tabular}{c>{\centering}p{1.5cm}>{\centering}p{1cm}>{\centering}p{1.2cm}>{\centering}p{1.4cm}>{\centering}p{5.cm}}
\hline
\noalign{\vskip3pt}
{Case} & { RGV's capacity} & { Energy\linebreak cost} & { Travel distance} & { Completion time} & { Vertex service order}\tabularnewline[3pt]
\hline
\noalign{\vskip3pt}
\multirow{3}{*}{{ (a)}} & { 1} & { 50.47} & { 38} & { 63.78} & { 4$\rightarrow$11$\rightarrow$7$\rightarrow$14$\rightarrow$3$\rightarrow$10$\rightarrow$5$\rightarrow$12$\rightarrow$1\linebreak$\rightarrow$8$\rightarrow$6$\rightarrow$13$\rightarrow$2$\rightarrow$9}\tabularnewline[3pt]
\noalign{\vskip3pt}
 & { 2/3/4/6} & { 42.63} & { 30} & { 55.64} & { 3$\rightarrow$10$\rightarrow$4$\rightarrow$11$\rightarrow$7$\rightarrow$5$\rightarrow$14$\rightarrow$12$\rightarrow$1\linebreak$\rightarrow$8$\rightarrow$6$\rightarrow$2$\rightarrow$13$\rightarrow$9}\tabularnewline[3pt]
\noalign{\vskip3pt}
 & { 5} & { 42.63} & { 30} & { 55.64} & { 1$\rightarrow$8$\rightarrow$6$\rightarrow$2$\rightarrow$13$\rightarrow$9$\rightarrow$3$\rightarrow$10$\rightarrow$4\linebreak$\rightarrow$11$\rightarrow$7$\rightarrow$5$\rightarrow$14$\rightarrow$12}\tabularnewline[3pt]
\noalign{\vskip3pt}
\multirow{1}{*}{{ (b)}} & { 2/3/4/5/6} & { 55.86} & { 36} & { 65.58} & { 1$\rightarrow$4$\rightarrow$10$\rightarrow$2$\rightarrow$13$\rightarrow$11$\rightarrow$8$\rightarrow$3$\rightarrow$17\linebreak$\rightarrow$12$\rightarrow$5$\rightarrow$14$\rightarrow$6$\rightarrow$15$\rightarrow$9$\rightarrow$7$\rightarrow$18$\rightarrow$16}\tabularnewline[3pt]
\hline
\end{tabular}}
\end{table*}

\subsection{Evaluating the valid inequalities and selecting a size for the rolling
horizon}

This subsection evaluates the usefulness of the valid inequalities
introduced in Section \ref{sec: Valid-constraints} for solving the
PDP-TSLU and meanwhile determines an appropriate size for the rolling
horizon to enable online decisions in a dynamic environment.

The valid inequalities (\ref{vc - subtour elimination})-(\ref{eq:vc - FIFO3b})
are added as redundant constraints to the PDP-TSLU. More precisely,
the valid inequalities (\ref{vc - subtour elimination}) with a subtour
size of 2 are added, and the special cases of constraints (\ref{vc - lifted subtour elimination1})-(\ref{eq:vc - Dkb})
are added:
\begin{itemize}
\item the constraints (\ref{vc - lifted subtour elimination1}) with $S$
equal to $\{i,j\}$, $\{i,j+n\}$, and $\{i,i+n,j\}$,
and the constraints (\ref{vc - lifted subtour elimination2}) with
$S$ equal to $\{i+n,j+n\}$, $\{i,j+n\}$, and $\{i,i+n,j+n\}$;
\item the constraints (\ref{eq:vc - Dka}) and (\ref{eq:vc - Dkb}) with
$k=3$: $x_{i+n,j}+x_{ji}+x_{i,i+n}+x_{j+n,i+n}+2x_{j,i+n}\le2$
and $x_{i,i+n}+x_{i+n,j+n}+x_{j+n,i}+x_{ij}+2x_{i,j+n}\le2$.
\end{itemize}
These simple constraints were once used in solving TSPPDL and TSPPDF
\citep{cordeau2010branch,cordeau2010branchFIFO}. In the meantime, the valid inequalities (\ref{eq: vc - LIFO1a})-(\ref{eq:vc - FIFO3b}) unique to PDP-TSLU
each has a polynomial cardinality and are added as redundant constraints
to the PDP-TSLU directly.

For convenience of investigation, we classify the valid inequalities (\ref{vc - subtour elimination})-(\ref{eq:vc - Dkb}) inherited from a generic PDP as \textit{Group 1}, and
(\ref{eq: vc - LIFO1a})-(\ref{eq:vc - LIFO3b}) derived from LIFO and CLO service restrictions as \textit{Group 2}, and (\ref{vc - FIFO1a})-(\ref{eq:vc - FIFO3b}) derived from FIFO and CLO service restrictions as \textit{Group
3}. The usefulness of including various combinations of these three groups
of valid inequalities for solving PDP-TSLU with/without TW constraints
is evaluated via 20 random instances. In each instance, the work area
consists of 10 pairs of face-to-face stations and the TW constraints (if present)
are generated as delivery deadlines. The deadline for request $i$
is a random variable following a uniform distribution within $[T_{i}-15,\,T_{i}]$,
where $T_{i}$ is the completion time of request $i$ for the same
instance in the absence of TW constraints.

The average CPU times to solve the instances are summarized in Table \ref{tab:test valid inequalities}. We observe that in general, use of any of the three groups of valid inequalities is able to reduce the computational time, which is however least likely when valid inequalities of Group 1 are applied. This implies that the inequalities inherited from a generic PDP is not as useful as those derived from the specific PDP-TSLU, and hence the single inclusion of the Group 1 inequalities is not recommended. Instead, a combination of the 1-3 or 2-3 groups of valid inequalities is able to reduce the computational time in both simulation cases, either with or without TW constraints.

\begin{table*}
{\footnotesize \caption{{ CPU solution time (in second) when different groups (grps.)
of valid inequalities were used.\label{tab:test valid inequalities}}}
}{\footnotesize \par}

\centering{}%
{\scriptsize
\begin{tabular}{>{\raggedright}m{1.4cm}|>{\raggedright}m{1.4cm}ccccccccccccc}
\hline
\noalign{\vskip3pt}
 & {$\left|P\right|$} & \multicolumn{2}{c}{{7}} &  & \multicolumn{2}{c}{{8}} &  & \multicolumn{2}{c}{{9}} &  & \multicolumn{2}{c}{{10}} &  & \multirow{2}{*}{\textit{Average}}\tabularnewline[3pt]
\cline{3-4} \cline{6-7} \cline{9-10} \cline{12-13}
\noalign{\vskip3pt}
 & {$Q$} & {2} & {4} &  & {2} & {4} &  & {2} & {4} &  & {2} & {4} &  & \tabularnewline[3pt]
\hline
\noalign{\vskip3pt}
 & {None} & {9.5} & {17.7} &  & {42.9} & {198.6} &  & {95.1} & {438.8} &  & {316.2} & {1120.2} &  & \textit{279.9}\tabularnewline[3pt]
\noalign{\vskip3pt}
\multirow{1}{1.4cm}{} & \multirow{1}{1.4cm}{{Grp. 1}} & {9.2} & {16.9} &  & {38.8} & {151.5} &  & {111.1} & {658.3} &  & {208.6} & {2061.7} &  & \textit{407.0}\tabularnewline[3pt]
\noalign{\vskip3pt}
\multirow{1}{1.4cm}{{Without TW cons.}} & \multirow{1}{1.4cm}{{Grp. 2}} & {9.0} & {18.0} &  & {26.6} & {131.1} &  & {110.0} & {352.1} &  & {226.0} & {1179.6} &  & \textit{256.6}\tabularnewline[3pt]
\noalign{\vskip3pt}
\multirow{1}{1.4cm}{} & \multirow{1}{1.4cm}{{Grp. 3}} & {6.0} & {18.4} &  & {26.2} & {134.3} &  & {87.6} & {295.6} &  & {227.0} & {1000.9} &  & \textit{224.5}\tabularnewline[3pt]
\noalign{\vskip3pt}
\multirow{1}{1.4cm}{} & \multirow{1}{1.4cm}{{Grps. 2-3}} & {6.6} & {17.0} &  & {31.3} & {132.8} &  & {82.4} & {237.9} &  & {272.2} & {645.7} &  & \textit{178.2}\tabularnewline[3pt]
\noalign{\vskip3pt}
 & {Grps. 1-3} & {7.5} & {16.1} &  & {35.7} & {144.2} &  & {95.1} & {414.1} &  & {283.9} & {972.1} &  & \textit{246.1}\tabularnewline[3pt]
\hline
\noalign{\vskip3pt}
 & {None} & {9.5} & {19.3} &  & {42.3} & {185.5} &  & {149.2} & {837.6} &  & {387.9} & {2691.6} &  & \textit{540.4}\tabularnewline[3pt]
\noalign{\vskip3pt}
\multirow{1}{1.4cm}{} & \multirow{1}{1.4cm}{{Grp. 1}} & {10.8} & {21.6} &  & {44.7} & {200.1} &  & {207.8} & {533.0} &  & {530.1} & {2301.6} &  & \textit{481.2}\tabularnewline[3pt]
\noalign{\vskip3pt}
\multirow{1}{1.4cm}{} & \multirow{1}{1.4cm}{{Grp. 2}} & {8.8} & {18.9} &  & {39.3} & {201.4} &  & {135.9} & {596.8} &  & {400.8} & {1801.6} &  & \textit{400.4}\tabularnewline[3pt]
\noalign{\vskip3pt}
\multirow{1}{1.4cm}{{With TW cons.}} & \multirow{1}{1.4cm}{{Grp. 3}} & {6.5} & {19.4} &  & {31.4} & {207.5} &  & {110.7} & {600.0} &  & {317.7} & {1079.5} &  & \textit{296.6}\tabularnewline[3pt]
\noalign{\vskip3pt}
\multirow{1}{1.4cm}{} & \multirow{1}{1.4cm}{{Grps. 2-3}} & {7.3} & {20.0} &  & {34.6} & {165.3} &  & {167.8} & {790.4} &  & {286.1} & {1776.2} &  & \textit{406.0}\tabularnewline[3pt]
\noalign{\vskip3pt}
\multirow{1}{1.4cm}{} & \multirow{1}{1.4cm}{{Grps. 1-3}} & {7.1} & {18.5} &  & {37.4} & {173.9} &  & {146.6} & {668.1} &  & {266.9} & {1086.9} &  & \textit{300.7}\tabularnewline[3pt]
\hline
\end{tabular}}
\end{table*}

The simulation results also indicate that using a rolling-horizon size of
eight would allow the control system to recompute a routing solution within
one minute when the RGV has a capacity of 2 and about three minutes
when the RGV has a capacity of 4. This suggests that a rolling-horizon
size of eight is reasonable for rendering
online routing solutions in a dynamic environment.

\subsection{Evaluating the proposed approach for handling dynamic PDP-TSLU}

Random instances with and without TW constraints are generated to evaluate
the rolling-horizon approach for treating dynamic PDP-TSLUs. In each instance, 50 PD requests are randomly generated in a range of 20 pairs of face-to-face stations. With each request is associated an
arrival time which follows a Poisson distribution with a mean of 0.5
as simulates the stochastic arrival practice.

In the case with TW constraints, each PD request is associated with
a delivery deadline, tight or loose.
Tight deadlines follow a uniform distribution in the range of [50,
80], and loose deadlines follow a uniform distribution
in the range of [150, 200]. In each
pickup queue the request joining first is assigned with a closer deadline. Thus the deadline in each pickup queue is sorted in an increasing
order from head to the rear. To simulate schedulable practice, one out
of three of the PD requests are imposed with tight deadlines and the
others with loose ones.

In each instance, the total energy cost is computed by applying each of the two routing approaches: the rolling-horizon
approach and a rule-based approach. The rolling-horizon approach
was introduced in Section \ref{sec: rolling-horizon-appro}, and is implemented with a rolling-horizon size of eight and by including valid inequalities
of Groups 1-3 into the problem model. In contrast,
the rule-based approach is heuristic and it works as follows.

\textit{Rule-based approach}: In the absence of TW constraints, all PD requests are queued in each station by their arrival times. The pickup service follows the priority rule of ``Earliest
Arrival First'', in which the RGV will first load
the container with the earliest arrival time and follow up with next
arrival container without causing unloading conflict with previously loaded
containers, until the RGV is full. The RGV will then deliver all loaded
containers before picking up any new ones. If delivery requests are assigned
with TW constraints, the priority rule base is shifted from arrival
time to deadline to ensure compliance of the TW constraints. The popular heuristics
of ``Earliest Due First'' is thus adopted,
in which the RGV keeps loading containers with closest deadlines as long as no service conflict arises until the RGV is full. Similarly, the RGV will
not pick up new containers until all loaded ones are delivered.
The complete priority rule mimics the method used in the current AFHS under consideration, and similar
rules had been reported in the literature \cite{roodbergen2009survey,lee1999dispatching}.
This rule-based approach updates the routing decision whenever a new PD request is issued,
and the simple algorithm is able to recompute a routing decision
in real time for a problem having up to fifty requests.

The two approaches are applied to solving random instances of dynamic
PDP-TSLUs without and with TW constraints. The
average simulation results are summarized in Table \ref{tab:dynamic-result-Energy-costs-and}.
We observe that, compared to the rule-based approach, the rolling-horizon approach is able to reduce energy cost by up to 15\%. However, the percentage saving decreases as the RGV's capacity increases or when TW constraints are imposed on the PD requests. We also note that the rolling-horizon approach
has an extra advantage of being able to render feasible decisions
in the presence of TW constraints, in which case the rule-based approach
is often hard to offer.

Another interesting observation is that, when the RGV's capacity
is increased from 2 to 3, both approaches achieve lower energy consumption, but the marginal energy savings are
much less than the naive estimate of $\left|\frac{\nicefrac{n}{3}-\nicefrac{n}{2}}{\nicefrac{n}{2}}\right|\times100\%\thickapprox33.33\%$.
This occurs because the conflict-free service constraints and the weight-dependent energy consumption prevent the RGV exploiting its full capacity. The insight
can be useful to warehouse designers, helping them determine appropriate
capacities for RGVs which work in a similar environment by evaluating
the trade-off between extra construction and operational cost it would
take and the actual benefit it would bring.

\begin{table*}
{ \caption{{ \label{tab:dynamic-result-Energy-costs-and}Total energy costs
for operating an RGV.}}
}{ \par}

\centering{}%
{\scriptsize
\begin{tabular}{lccccccccccc}
\hline
\noalign{\vskip3pt}
\multirow{2}{*}{} & \multicolumn{5}{c}{{ Energy cost (without TW constraints)}} &  & \multicolumn{5}{c}{{ Energy cost (with TW constraints)}}\tabularnewline[3pt]
\cline{2-6} \cline{8-12}
\noalign{\vskip3pt}
 & \multicolumn{2}{>{\centering}p{1.6cm}}{{ Rule-based\linebreak approach}} &  & \multicolumn{2}{>{\centering}p{1.9cm}}{{ Rolling-horizon\linebreak approach}} &  & \multicolumn{2}{>{\centering}p{1.6cm}}{{ Rule-based\linebreak approach}} &  & \multicolumn{2}{>{\centering}p{1.9cm}}{{ Rolling-horizon\linebreak approach}}\tabularnewline[3pt]
\cline{2-3} \cline{5-6} \cline{8-9} \cline{11-12}
\noalign{\vskip3pt}
{ Instance} & { $Q=2$} & { $Q=3$} &  & { $Q=2$} & { $Q=3$} &  & { $Q=2$} & { $Q=3$} &  & { $Q=2$} & { $Q=3$}\tabularnewline[3pt]
\hline
\noalign{\vskip3pt}
{ a} & { 1335.7} & { 1222.0} &  & { 1067.2} & { 910.4} &  & { 1386.8} & { 1352.4} &  & { 1173.6} & { 1076.0}\tabularnewline[3pt]
\noalign{\vskip3pt}
{ b} & { 1324.0} & { 1169.7} &  & { 1112.9} & { 1128.6} &  & { 1504.3} & { 1239.7} &  & { 1135.8} & { 1214.3}\tabularnewline[3pt]
\noalign{\vskip3pt}
{ c} & { 571.3} & { 583.7} &  & { 552.7} & { 477.3} &  & { 592.9} & { 595.9} &  & { 569.4} & { 505.7}\tabularnewline[3pt]
\noalign{\vskip3pt}
{ d} & { 618.4} & { 555.7} &  & { 516.5} & { 464.8} &  & { 651.7} & { 607.5} &  & { 559.6} & { 471.5}\tabularnewline[3pt]
\noalign{\vskip3pt}
{ e} & { 641.9} & { 563.5} &  & { 567.4} & { 540.0} &  & { 645.9} & { 571.3} &  & { 639.9} & { 561.3}\tabularnewline[3pt]
\noalign{\vskip3pt}
\multicolumn{1}{l}{\textit{ Average}} & \textit{ 898.3} & \textit{ 818.9} &  & \textit{ 763.3} & \textit{ 704.2} &  & \textit{ 956.3} & \textit{ 873.4} &  & \textit{ 815.7} & \textit{ 765.8}\tabularnewline[3pt]
\noalign{\vskip3pt}
\multicolumn{2}{l}{\textit{ Percentage saving}} &  &  & \textit{ 15.03\%} & \textit{ 14.01\%} &  &  &  &  & \textit{ 14.7\%} & \textit{ 12.32\%}\tabularnewline[3pt]
\hline
\end{tabular}}
\end{table*}

\section{Conclusions} \label{sec: conclusions}

This work studied an RGV routing problem in an automated freight handling system. The problem was formulated as an MILP that aims to minimize energy consumption for an RGV to complete pickup and delivery tasks under conflict-avoidance and time window constraints. The energy consumption takes account of routing-dependent gross weight and also dynamics of the RGV, and the conflict-avoidance constraints guarantee conflict-free service under two-sided loading/unloading operations. Arc reduction and valid inequalities were exploited from the problem structure to enhance the MILP solution process. The static problem model and solution approach were integrated with a rolling-horizon approach to handle
the dynamic routing problem where air cargo enters and departs from the system dynamically in time. Numerical experiments suggest that the proposed routing strategy is able to route an RGV to transport cargo with an energy cost up to 15\% lower than a heuristic method implemented in current practice.

The current routing problem assumes a single RGV serving the system. In practice, however, multiple RGVs may be employed which work on a single track or different tracks that handle coupled transport tasks, and meanwhile a container may be relayed to its destination via multiple RGVs. The routing problem thus becomes more complex, and collective routing of the RGVs and containers will be required for achieving optimal energy efficiency under service quality and feasibility constraints. This constitutes interesting and challenging work for the future research.

\section*{Acknowledgment}

This work was supported in part by ATMRI under Grant M4061216.057, by MOE AcRF under Tier 1 grant RG 33/10 M4010492.050 and by NTU under startup grant M4080181.050.

\singlespacing
\section*{Appendices}
{\footnotesize
\renewcommand{\thesubsection}{\Alph{subsection}}

\subsection{Equivalent forms of the CFI and CLO service constraints} \label{apdix: CFI-alternative}
\begin{lem}
The following CFI service constraints are equivalent:
\begin{align*}
\text{(See (\ref{eq:CFI}))} \sum_{i:(i,j)\in A'}y_{ij}^{k}=0,\,\,&\forall j\in P^{3},\, k\in P^{1}\\
\Longleftrightarrow\sum_{i:(i,j+n)\in A'}y_{i,j+n}^{k}\le\sum_{i:(i,j)\in A'}y_{ij}^{k},\,\,&\forall j\in P^{1},\, k\in P^{3},
\end{align*}
and the following CLO service constraints are equivalent:
\begin{align*}
(\text{See (\ref{eq:CLO})}) \sum_{i:\,(i,\, j+n)\in A'}y_{i,\, j+n}^{k}=0, \quad &\forall j\in P^{4}, k\in P^{1} \\
\Longleftrightarrow \sum_{i:\,(i,\, j)\in A'}y_{ij}^{k}\le\sum_{i:\,(i,\, j+n)\in A'}y_{i,\, j+n}^{k}, \quad &\forall j\in P^{1}, k\in P^{4}.
\end{align*}
\end{lem}
\begin{proof}
The proofs of the two equivalences follow similar reasoning. For brevity, we only detail the proof for equivalence between the two CFI service constraints. We prove it by contradiction.

``$\Rightarrow$'':
It is sufficient to show that, given an arbitrary pair $j\in P^{3},\, k\in P^{1}$,
it must have $\sum_{i:(i,k+n)\in A'}y_{i,k+n}^{j}\le\sum_{i:(i,k)\in A'}y_{ik}^{j}$.
Suppose this is not true, \ie, there exists a pair $j\in P^{3},\, k\in P^{1}$ such that $\sum_{i:(i,k+n)\in A'}y_{i,k+n}^{j}>\sum_{i:(i,k)\in A'}y_{ik}^{j}$.
This implies that $\sum_{i:(i,k+n)\in A'}y_{i,k+n}^{j}=1$ and $\sum_{i:(i,k)\in A'}y_{ik}^{j}=0$,
which means that request $k$ is picked up before service of request
$j$ and delivered during service of request $j$. Equivalently, this
means that request $j$ is picked up during the service of request $k$, $\ie,$ $\sum_{i:\,(i,\, j)\in A'}y_{ij}^{k}=1$. This
is contradictory to the given condition and hence proves the sufficiency.

``$\Leftarrow$'': It is proved alike. Given an arbitrary pair
$j\in P^{1},\, k\in P^{3}$, suppose $\sum_{i:(i,k)\in A'}y_{ik}^{j}=1$.
Thus, $\sum_{i:(i,k)\in A'}y_{ik}^{j}\ge\sum_{i:(i,k+n)\in A'}y_{i,k+n}^{j}$,
which means that request $k$ can either
be picked up and delivered, or be picked up but not delivered during
the service of request $j$. This implies that in certain circumstance the request $j$ can be delivered but not picked up during the service of request $k$, \ie, it is possible to have $\sum_{i:(i,j+n)\in A'}y_{i,j+n}^{k}=1>0=\sum_{i:(i,j)\in A'}y_{i,j}^{k}$. This
is contradictory to the given condition and hence proves the necessity.\end{proof}

\subsection{Big-$M$ formulation of (\ref{eq:c - time consistency})} \label{apdix: big-M formulation}

The indicator constraint (\ref{eq:c - time consistency}) can
be linearized by using regular big-$M$ formulation as
\begin{equation}
\begin{aligned}
b_{j}\ge b_{i}+s_{i}+t_{ij}-\eta_{ij}(1-x_{ij}), \quad &  \forall  (i,j)\in A,\\
w_{j}\ge w_{i}+q_{j}-\rho_{ij}(1-x_{ij}), \quad &  \forall  (i,j)\in A,
\end{aligned}
\end{equation}
where $\eta_{ij}$ and $\rho_{ij}$ are large enough constants such
that the right hand sides of the inequalities are lower bounds of
$b_{j}$ and $w_{j}$, respectively. Specific lower bounds can be derived by using the constraints (\ref{eq:c - completion time}), (\ref{eq:c - time window}), (\ref{eq:c - load bound}) and (\ref{eq:c - b and w}), resulting in feasible $\eta_{ij}$ and $\rho_{ij}$
as summarized in Table \ref{tab:Specifications-of-eta-rho}, where
$l_{0}=q_{0}=q_{2n+1}\triangleq 0$. The detailed derivation is given below.

\begin{table}[H]
\caption{\footnotesize Feasible constant pairs $(\eta_{ij},\,\rho_{ij})$ ($\forall (i,j)\in A$)} \label{tab:Specifications-of-eta-rho}
\centering{}%
\footnotesize
\begin{tabular}{ccc}
\hline
\noalign{\vskip3pt}
 & $i\in P\cup\{0\}$ & $i\in D$\tabularnewline[3pt]
\hline
\noalign{\vskip3pt}
$j\in P$ & $(l_{i}-s_{i+n}-t_{i,i+n}+t_{ij},Q)$ & $(l_{i-n}+t_{ij},Q+q_{i})$\tabularnewline[3pt]
\noalign{\vskip3pt}
$j\in D\cup\{2n+1\}$ & $(l_{i}-s_{i+n}-t_{i,i+n}+t_{ij}-e_{j-n},Q+q_{j})$ & $(l_{i-n}+t_{ij}-e_{j-n},Q+q_{i}+q_{j})$\tabularnewline[3pt]
\hline
\end{tabular}
\end{table}

Valid reformations of the time and loading consistency constraints
in (\ref{eq:c - time consistency}) require that $\eta_{ij}\ge b_{i}+s_{i}+t_{ij}-b_{j}$
and $\rho_{ij}\ge w_{i}+q_{j}-w_{j}$ for any $(i,j)\in A$. It
is sufficient to set $\eta_{ij}$ and $\rho_{ij}$ as respective upper
bounds of $\underline{\eta}_{ij}\triangleq b_{i}+s_{i}+t_{ij}-b_{j}$
and $\underline{\rho}_{ij}\triangleq w_{i}+q_{j}-w_{j}$ for a given
$(i,j)$. Based on (\ref{eq:c - completion time}), (\ref{eq:c - time window}),
(\ref{eq:c - load bound}) and (\ref{eq:c - b and w}), such upper
bounds are derived for four cases of arcs as follows.

\emph{Case a}: $i\in P\cup\{0\}$, $j\in P$. It follows that $\underline{\eta}_{ij}\le b_{i+n}-t_{i,i+n}+t_{ij}-b_{j}\le l_{i}-s_{i+n}-t_{i,i+n}+t_{ij}$,
and that $\underline{\rho}_{ij}\le w_{i}+q_{j}-q_{j}\le Q$.

\emph{Case b}: $i\in P\cup\{0\}$, $j\in D\cup\{2n+1\}$. It
follows that $\underline{\eta}_{ij}\le b_{i+n}-t_{i,i+n}+t_{ij}-b_{j}\le l_{i}-s_{i+n}-t_{i,i+n}+t_{ij}-e_{j-n}$,
and that $\underline{\rho}_{ij}\le w_{i}+q_{j}\le Q+q_{j}$.

\emph{Case c}: $i\in D$, $j\in P$. It follows that $\underline{\eta}_{ij}\le l_{i-n}+t_{ij}-b_{j}\le l_{i-n}+t_{ij}$,
and that $\underline{\rho}_{ij}\le w_{i}+q_{j}-q_{j}\le Q+q_{i}$.

\emph{Case d}: $i\in D$, $j\in D\cup\{2n+1\}$. It follows
that $\underline{\eta}_{ij}\le l_{i-n}+t_{ij}-b_{j}\le l_{i-n}+t_{ij}-e_{j-n}$,
and that $\underline{\rho}_{ij}\le w_{i}+q_{j}\le Q+q_{i}+q_{j}$.

\bibliographystyle{elsarticle-num}
\bibliography{RGV_routing}
}

\end{document}